\RequirePackage{fix-cm}
\documentclass[smallextended]{svjour3}     
\smartqed 
\usepackage{graphicx}
\usepackage{color}
\usepackage{subfigure}
\usepackage{float}
\usepackage{amsmath,amsfonts,amssymb}
\usepackage{textcase}            
\usepackage{bm}                        

\newcommand{\T}{\intercal}	
\newcommand{\tf}{ {t_\mathrm{f}} }	
\newcommand{\ts}{ {t_\mathrm{s}} }	
\newcommand{\td}{\mathrm{d}}	
\newcommand{\s}{\mathrm{s}}

\renewcommand{\vec}[1]{\MakeTextLowercase{\bm{#1}}}
\newcommand{\mat}[1]{\MakeTextUppercase{\bm{#1}}}
\newcommand{\op}[1]{\MakeTextUppercase{\bm{#1}}}
\newcommand{\set}[1]{\mathcal{\MakeTextUppercase{#1}}}

\newcommand{\real}{\mathbb{R}}
\newcommand{\al}[2]{\alpha_{#1}^{(#2)}}

\newcommand{\mL}{\mathcal{L}}
\newcommand{\mP}{\mathcal{P}}
\newcommand{\mO}{\mathcal{O}}
\newcommand{\mN}{\mathcal{N}}
\newcommand{\mI}{\mathcal{I}}
\DeclareMathOperator*{\NBV}{NBV}
\DeclareMathOperator{\e}{e} 
\newcommand{\HLag}{\mL}
\newcommand{\Lag}{\mL}
\newcommand{\eLag}{\hat{\mL}}

\newcommand{\bse}{\begin {subequations}}
\newcommand{\ese}{\end {subequations}}

\newcommand{\nrm}[2]{\left|\left| #1 \right|\right|_{#2}}
\newcommand{\norm}[1]{\left|\left| #1 \right|\right|}
\newcommand{\abs}[1]{\left| #1 \right|}

\begin{document}

\title{Approximation of weak adjoints by reverse automatic differentiation of
BDF methods}

\titlerunning{Approximation of weak adjoints by reverse AD of BDF methods} 

\author{D\"orte Beigel \and Mario S. Mommer \and Leonard Wirsching \and Hans
Georg Bock}

\authorrunning{D. Beigel, M.S. Mommer, L. Wirsching, H.G. Bock} 

\institute{
Interdisciplinary Center for Scientific Computing (IWR), Heidelberg
University.
Im Neuenheimer Feld 368, 69120 Heidelberg, Germany.
Fax: +49-6221-54 5444.\\
\email{doerte.beigel@iwr.uni-heidelberg.de}, Tel.: +49-6221-54 8896
(corresponding).\\
\email{\{mario.mommer, leonard.wirsching, bock\}@iwr.uni-heidelberg.de}.
}

\date{September 13, 2011}

\maketitle

\begin{abstract}

With this contribution, we shed light on the relation between the
discrete adjoints of multistep backward differentiation formula (BDF) methods
and the solution of the adjoint differential equation. To this end, we develop a
functional-analytic framework based on a constrained variational problem and
introduce the notion of weak adjoint solutions. We devise a finite element
Petrov-Galerkin interpretation of the BDF method together with its discrete
adjoint scheme obtained by reverse internal numerical differentiation.
We show how the finite element approximation of the weak adjoint is computed by
the discrete adjoint scheme and prove its asymptotic convergence in
the space of normalized functions of bounded variation. We also obtain
asymptotic convergence of the discrete adjoints to the classical adjoints
on the inner time interval. Finally, we give
numerical results for non-adaptive and fully adaptive BDF schemes. The presented
framework opens the way to carry over the existing theory on global error
estimation techniques from finite element methods to BDF methods.

\end{abstract}

\section{Introduction}
\label{sec:intro}

Consider a nonlinear initial value problem (IVP) in ordinary differential
equations (ODE) with sufficiently smooth right hand side $\vec{f}: [ \ts,\tf ]
\times \real^d \rightarrow \real^d $
\bse\label{eq:IVP}
\begin{align}
\dot{\vec{y}}(t) &= \vec{f}(t, \vec{y}(t)),\quad t\in
(\ts,\tf] \label{eq:IVP_ode}\\
\vec{y}(\ts) &= \vec{y}_\text{s}. \label{eq:IVP_ic}
\end{align}
\ese
Consider also a differentiable criterion of interest $J$ depending on
the final state $\vec{y}(\tf)$ of the solution of \eqref{eq:IVP}. This
is relevant whenever one is not interested in the whole
solution trajectory $\vec{y}(t)$ or even the final state $\vec{y}(\tf)$, but
only in a functional output of these quantities. Note that by
standard
reformulations (cf. \cite[p.93]{Hartman2002}, \cite[p.25]{Berkovitz1974}) this
setting also captures the cases of a parameter-dependent right hand side
$\vec{f}(t,\vec{y},\vec{p})$ and a criterion of interest of Bolza type
$J(\vec{y})= \int_\ts^\tf J_1(\vec{y}(t),\vec{p})\td t + J_2(\vec{y}(\tf))$. 

The adjoint differential equation corresponding to the evaluation of
$J(\vec{y}(\tf))$ in the solution of \eqref{eq:IVP} is (see Section
\ref{sec:Hadj})
\bse \label{eq:aIVP}
\begin{align}
\dot{\vec{\lambda}}(t) &= - \vec{f}^\T_{\vec{y}}(t,\vec{y}(t))\vec{\lambda}(t),
\quad t\in (\tf,\ts] \label{eq:aIVP_ode} \\
\vec{\lambda}(\tf) &= J^\prime(\vec{y}(\tf))^\T.\label{eq:aIVP_ic}
\end{align}
\ese
The adjoint solution describes the dependency of $J(\vec{y}(\tf))$ on
disturbances of the nominal solution $\vec{y}(t)$. Therefore, it is of great
importance in the solution of optimal control problems. For example, in indirect
approaches based on the Pontryagin minimum principle, \eqref{eq:aIVP} appears as
part of the optimality conditions. 

For an approximation of the solution of \eqref{eq:IVP}, the
solution of the adjoint differential equation \eqref{eq:aIVP} can be
computed in two different ways, the continuous adjoint approach or the discrete
adjoint approach.  The former solves the adjoint differential equation by
numerical integration, see for example \cite{Cao2002}. Whereas the latter
applies automatic differentiation techniques to the numerical integration
scheme. This approach, firstly presented in \cite{Bock1981}, is known as
internal numerical differentiation (IND). It has significant advantages
in direct derivative-based approaches for the solution of optimal control
problems that use integrators, e.g. direct single and multiple shooting.

In the case of Runge-Kutta methods, the discrete adjoint scheme generated by
adjoint IND is itself a Runge-Kutta scheme for the adjoint differential
equation \eqref{eq:aIVP}, and thus gives a convergent approximation to the
adjoint solution \cite{Bock1981,Walther2007}. In the case of continuous and
discontinuous
Galerkin methods applied to \eqref{eq:IVP}, the discrete adjoint scheme yields
an approximation to the solution of \eqref{eq:aIVP} (see e.g.
\cite{Johnson1988}). The discrete adjoints of discontinuous Galerkin methods for
compressible Navier-Stokes equations are, for example, considered in
\cite{Hartmann2007}.

The situation becomes significantly more complex in the case of multistep
methods, as the discrete adjoint schemes of linear multistep methods
(LMM) are generally \textit{not} consistent with the adjoint differential
equation \eqref{eq:aIVP}. But they still provide approximations of the
sensitivities $J^\prime(\vec{y}(\tf)) \frac{\partial \vec{y}(\tf)}{\partial
\vec{y}_\s}$ at the initial time $\ts$ that converge with the rate of the
nominal LMM \cite{Bock1987,Sandu2008}. Due to this property, the multistep BDF
method and its discrete adjoint scheme are used successfully in direct methods
for the solution of optimal control problems, e.g. in direct multiple
shooting \cite{Bock1984,Albersmeyer2010d}.

In this contribution, we focus on the relation between the discrete adjoints of
variable-order variable-stepsize BDF methods and the adjoints defined by
\eqref{eq:aIVP}. To this end, we construct a suitable constrained variational
problem (CVP) in a Banach space setting using the duality pairing between the
space of continuous functions and its dual, the space of normalized
functions of bounded variation. It turns out that the adjoint of a stationary
point of this CVP is the normalized integral of the solution of the Hilbert
space adjoint differential equation \eqref{eq:aIVP}. Motivated by PDE
nomenclature, we will call it a weak solution of \eqref{eq:aIVP} or shortly
weak adjoint. We apply Petrov-Galerkin
techniques, and show that with the appropriate choice of basis functions the
infinite-dimensional optimality conditions of the CVP are approximated by the
BDF method and its discrete adjoint scheme obtained by
adjoint internal numerical differentiation of the nominal BDF scheme. In
particular, we obtain that discretization and optimization commute in this
Banach space setting. Finally, we prove that the finite element approximation of
the weak adjoint, which can be computed by a simple post-processing of the
discrete adjoints, converges to the weak adjoint on the entire time
interval. This result is based on the linear convergence of the discrete
adjoints to the solution of \eqref{eq:aIVP} on the inner time interval which is
shown as well. 

This paper is organized as follows. In Section \ref{sec:Hadj}  we derive the
adjoint differential equation as part of the optimality conditions of an
infinite-dimensional constrained variational problem in Hilbert spaces. The BDF
method and its discrete adjoint scheme generated by internal numerical
differentiation techniques are then described in Section \ref{sec:BDF_aBDF}. In
Section \ref{sec:B_CVP} we present the optimality conditions of the
constrained
variational problem embedded into the Banach space of all continuously
differentiable functions. After showing the well-posedness of the optimality
conditions and their relation to the Hilbert space optimality conditions, we
extend the setting to capture the space of all functions that are
continuous and piecewise
continuously differentiable. For the Petrov-Galerkin discretization of
Section \ref{sec:PGappr} we choose suitable finite element spaces that
yield equivalence between the discretized optimality conditions and the BDF
scheme together with its discrete adjoint scheme. In Section \ref{sec:conv} we
start by proving the convergence of the discrete adjoints to the solution
of the Hilbert space adjoint equation on the inner time interval. Using this
result, we show the convergence of the finite element approximation to the
weak adjoint solution. Section \ref{sec:numres} presents numerical
results on a nonlinear test case with analytic solutions.

\section{Initial value problems and their adjoints in a Hilbert space setting}
\label{sec:Hadj}

In this section, we derive the adjoint differential equation
in a Hilbert space functional-analytic setting. Our goal is to specify the
assumptions on the initial value
problem, to settle some notation, and to lay the groundwork for the
constructions that follow. In particular, we make explicit the connection
between the adjoint differential equation and the Lagrange multiplier of the
solution of a constrained variational problem in a Hilbert space setting based
on the Sobolev spaces usually found in finite element formulations.

\subsection{Existence, uniqueness and differentiability of the nominal solution}
\label{sec:EUDnomsol}

Assume that the right hand side $\vec{f}(t,\vec{y})$ of \eqref{eq:IVP} is
continuous on an open set $\set{D} \subset \real \times \real^d$ with
$(\ts,\vec{y}_\s)\in \set{D}$ and its first-order partial derivative
$\vec{f}_{\vec{y}}(t,\vec{y})$ is continuous on $\set{D}$. Thus, according to
the Picard-Lindel{\"o}f Theorem \cite{Hartman2002}, 
problem \eqref{eq:IVP} is well-posed in the sense of Hadamard, i.e. it admits a
unique solution depending continuously on the input data. Beyond that, the
solution $\vec{y}(t)$ is continuously differentiable on an open interval
$\set{I}$, see \cite{Hartman2002}, %
and we assume that $\tf$ is chosen such that $[\ts,\tf] \subset \set{I}$. Thus,
the solution $\vec{y}(t)$ of \eqref{eq:IVP} lies in the Banach space
$C^1[\ts,\tf]^d$ of all continuously differentiable functions from $[\ts,\tf]$
to $\real^d$ equipped with the usual norm. Furthermore, the solution
$\vec{y}(t)=\vec{y}(t;\ts,\vec{y}_\s)$ is continuously differentiable with
respect to $\vec{y}_\s$ and the derivatives $\vec{w}_i(t)=\partial
\vec{y}(t;\ts,\vec{y}_\s)/\partial (\vec{y}_\s)_i$ solve \cite{Hartman2002}
\bse\label{eq:fIVP}
\begin{align}
\dot{\vec{w}}_i(t) &= \vec{f}_{\vec{y}}(t,\vec{y}(t))\vec{w}_i(t),\quad t
\in(\ts,\tf] \label{eq:fIVP_ode}\\
\vec{w}_i(\ts) &= \vec{e}_i \label{eq:fIVP_ic}
\end{align}
\ese
where $\vec{e}_i$ is the $i$th unit vector, $i \in \{ 1,\dots,d\}$. Moreover,
$\vec{w}_i(t)$ exists uniquely and is continuously differentiable on $[\ts,\tf]$
since the partial derivative of the right hand side of \eqref{eq:fIVP} with
respect to $\vec{w}_i$ is continuous in $(t,\vec{w}_i)$. The residual of
\eqref{eq:IVP_ode}
\begin{align}
 \vec{\rho}(\vec{y}):= \dot{\vec{y}}(\cdot)-\vec{f}(\cdot,\vec{y}(\cdot))
\end{align}
lies in the Banach space $C^0[\ts,\tf]^d$ of all continuous functions from
$[\ts,\tf]$ to $\real^d$ equipped with the standard norm
$\nrm{\vec{g}}{C^0[\ts,\tf]^d} = \sum_{i=1}^d \nrm{g_i}{C^0[\ts,\tf]}$ where
$\nrm{g_i}{C^0[\ts,\tf]}=\max_{t \in [\ts,\tf]} \,\abs{g_i(t)}$.

\subsection{Lagrange multipliers and adjoint differential equations}
\label{sec:HLag}

The core of this section is the identification of the adjoint as the Lagrange
multiplier of a constrained optimization problem in a functional-analytic
setting. The ideas described here are of course not new. However, the
setting for the case of ordinary differential equations is fundamental for this
contribution. Since we have not found it in the literature, we include here a
detailed derivation.

Recall that functions in $C^0[\ts,\tf]^d$, restricted to the open interval
$(\ts,\tf)$, form a dense subset of the space $L^2(\ts,\tf)^d$ of all
quadratically Lebesgue-integrable functions. Similarly, recall that the subset
$C^1[\ts,\tf]^d$ is dense in the Sobolev space $H^1(\ts,\tf)^d$ of all
$L^2(\ts,\tf)^d$-functions with weak derivative in $L^2(\ts,\tf)^d$ (see
\cite[Ch.3]{Adams2003}). %
Furthermore, both spaces $L^2(\ts,\tf)^d$ and $H^1(\ts,\tf)^d$ are Hilbert
spaces.

Knowing this, we embed the initial value problem \eqref{eq:IVP} into an
optimization framework and derive the adjoint differential equation as part of
the first-order necessary optimality conditions. To this end, we consider the
constrained variational problem
\bse\label{eq:OCP}
\begin{align}
\min_{\vec{y}} \enspace & J(\vec{y}(\tf))\\
\operatorname{s.t.} \enspace &\dot{\vec{y}}(t) = \vec{f}(t,
\vec{y}(t)),\quad t\in
(\ts,\tf]  \\
&\vec{y}(\ts) = \vec{y}_\text{s}
\end{align}
\ese
which is equivalent to evaluating $J(\vec{y}(\tf))$ in the solution of
\eqref{eq:IVP}. Considering \eqref{eq:OCP} on the space $H^1(\ts,\tf)^d$,
the Hilbert space Lagrangian $\HLag: H^1(\ts,\tf)^d
\times L^2(\ts,\tf)^d \rightarrow \real$ of \eqref{eq:OCP} using the 
$L^2$-scalar product is
\begin{align}
\HLag(\vec{y},\vec{\lambda}):= J(\vec{y}(\tf)) - \int_\ts^\tf
\vec{\lambda}^\T(t)
\left[ \dot{\vec{y}}(t)-\vec{f}(t,\vec{y}(t))\right]\td t -
\vec{\lambda}^\T(\ts) \left[ \vec{y}(\ts)-\vec{y}_\s\right] \nonumber
\end{align}
where $\vec{\lambda}$ is the Lagrange multiplier. The optimality condition of
\eqref{eq:OCP} is based on the Fr\'{e}chet derivative of $\HLag$ at
$(\vec{y},\vec{\lambda})$ in direction $(\vec{w},\vec{\chi})$ which exists
due to Fr\'{e}chet differentiability of $J$ and \cite[Ch.0\S0.2.5]{Ioffe1979}
\begin{align}
\HLag^\prime&(\vec{y},\vec{\lambda})(\vec{w},\vec{\chi})= 
\HLag_{\vec{y}}(\vec{y},\vec{\lambda})(\vec{w}) + 
\HLag_{\vec{\lambda}}(\vec{y},\vec{\lambda})(\vec{\chi}) 
\nonumber\\
=&\left\lbrace 
\displaystyle J^\prime(\vec{y}(\tf)) \vec{w}(\tf) - \int_\ts^\tf
\vec{\lambda}^\T(t) \left[
\dot{\vec{w}}(t)-\vec{f}_{\vec{y}}(t,\vec{y}(t))\vec{w}(t) \right] \td t -
 \vec{\lambda}^\T(\ts) \vec{w}(\ts) \right\rbrace \nonumber\\
&+ \left\lbrace 
\displaystyle- \int_\ts^\tf \vec{\chi}^\T(t) \left[
\dot{\vec{y}}(t)-\vec{f}(t,\vec{y}(t))\right] \td t - \vec{\chi}^\T(\ts) 
\left[ \vec{y}(\ts)-\vec{y}_\s \right] 
\right\rbrace. \nonumber
\end{align}
The necessary condition for a stationary point $(\vec{y},\vec{\lambda})\in
H^1(\ts,\tf)^d \times L^2(\ts,\tf)^d$ of \eqref{eq:OCP} is that
$\HLag^\prime(\vec{y},\vec{\lambda})(\vec{w},\vec{\chi})=0$ holds for all
directions $(\vec{w},\vec{\chi})\in  H^1(\ts,\tf)^d \times L^2(\ts,\tf)^d$.
Choosing $\vec{w}=\vec{0} \in  H^1(\ts,\tf)^d $ and only
varying $\vec{\chi} \in L^2(\ts,\tf)^d$ the necessary condition reads
\begin{align}\label{eq:Hvarform}
\int_\ts^\tf \vec{\chi}^\T(t) \left[
\dot{\vec{y}}(t)-\vec{f}(t,\vec{y}(t))\right]\td t + \vec{\chi}^\T(\ts) 
\left[ \vec{y}(\ts)-\vec{y}_\s \right] = 0,\; \forall \vec{\chi}
\end{align}
which possesses the same unique solution $\vec{y}\in C^1[\ts,\tf]^d$ as
\eqref{eq:IVP}. Taking now $\vec{\chi}=\vec{0} \in L^2(\ts,\tf)^d$ and only
varying $\vec{w} \in H^1(\ts,\tf)^d$ one obtains using integration by parts 
\begin{align}
\left[J^\prime(\vec{y}(\tf))-\vec{\lambda}^\T(\tf)\right] \vec{w}(\tf)-
 \int^\ts_\tf \left[\dot{\vec{\lambda}}(t)+ \vec{f}^\T_{\vec{y}}(t,\vec{y}(t))
\vec{\lambda}(t)\right]^\T \vec{w}(t) \td t= 0,\; \forall \vec{w} \nonumber
\end{align}
which possesses the same solution as \eqref{eq:aIVP}. Under the assumptions of
Section \ref{sec:EUDnomsol}, the unique solution $\vec{\lambda}(t)$ of
\eqref{eq:aIVP} is continuously differentiable on $[\ts,\tf]$ and depends
continuously on $J^\prime(\vec{y}(\tf))^\T$.

\section{Efficient solution of initial value problems and sensitivity
generation}
\label{sec:BDF_aBDF}

We now review the numerical solution of ODEs using BDF methods, and the
corresponding sensitivity generation using automatic differentiation 
techniques. We briefly introduce BDF methods with an emphasis on the
trajectories they define as functions of time. Then, we show how to
obtain discrete adjoints in the BDF context, and review what is known so far
about their relation to the solution of \eqref{eq:aIVP}.

\subsection{Backward differentiation formula method}
\label{sec:BDF}

This section follows the lines of \cite[p.181ff and p.253f]{Shampine1994}.
Consider the backward differentiation formula method 
\bse\label{eq:BDF}
\begin{align}
\vec{y}_0 &= \vec{y}_\s \label{eq:BDF_ini}\\
\sum_{i=0}^{k_n} \al{i}{n} \vec{y}_{n+1-i} &= h_n
\vec{f}(t_{n+1},\vec{y}_{n+1}),\quad n=0,\dots , N-1
\label{eq:BDF_main}
\end{align}
\ese
with a self-starting procedure that begins with $k_0=1$ (implicit Euler) and
increases successively the order of the steps until the maximum order is
reached. Note that BDF methods are used up to order 6, since for higher order
they become unstable. In practical implementations both the stepsize $h_n$ and
the order $k_n$ are chosen adaptively to obtain better performance. The
numerical solution is computed at
discrete time points $\ts=t_0<t_1<\dots<t_N=\tf$ with $t_{n+1}=t_n+h_n$ and
$\vec{y}_n$ denotes the numerical approximation to the value $\vec{y}(t_n)$. The
coefficients $\al{i}{n}$ are determined by %
\begin{align}\label{eq:Lagpol}
\al{i}{n} = h_n \dot{L}_i^{(n)}(t_{n+1}), \text{
where } L_i^{(n)}(t) = \prod_{j=0,j\neq i}^{k_n} \frac{t-t_{n+1-j}}{t_{n+1-i} -
t_{n+1-j}}
\end{align}
are the fundamental Lagrangian polynomials. Thus, the coefficients depend on the
discrete time points and the order. In each step, the BDF method provides a
polynomial approximation to the solution $\vec{y}(t)$ of \eqref{eq:IVP} in a
natural way through the interpolation polynomial
\begin{align}
 \vec{y}(t)\big\vert_{t \in [t_n, t_{n+1}]} \approx 
\sum_{i=0}^{k_n} L_i^{(n)}(t)\;
\vec{y}_{n+1-i},
\end{align}
also known as \emph{dense output}. The composition of all these polynomials
gives a continuous and piecewise continuously differentiable approximation to
the solution $\vec{y}(t)$ on the whole time interval $[\ts,\tf]$.

\subsection{Adjoint differentiation of BDF integration schemes}
\label{sec:aBDF}

The basic idea of internal numerical differentiation \cite{Bock1981} is to
differentiate the discretization scheme used to obtain the nominal
approximations $\vec{y}_0,\vec{y}_1,\dots,\vec{y}_N$ for specified adaptive
components $h_n$ and $k_n$ using automatic differentiation
(AD) techniques either in forward or in adjoint mode. Adjoint IND was first
described in \cite{Bock1987} for Runge-Kutta integration schemes and later on in
\cite{Bock1994} for BDF methods. Applying adjoint IND to the BDF scheme
\eqref{eq:BDF} we obtain the discrete adjoint scheme
\bse\label{eq:aBDF}
\begin{align}
\al{0}{N-1} \vec{\lambda}_N - J^\prime(\vec{y}_N)^\T &= h_{N-1}
\vec{f}^\T_{\vec{y}}(t_N,\vec{y}_N)\vec{\lambda}_N \label{eq:aBDF_ini}\\
\sum_{0 \leq i \leq N-1-n \atop i\leq k_{\max} } \al{i}{n+i}
\vec{\lambda}_{n+1+i}
&=  h_n \vec{f}^\T_{\vec{y}}(t_{n+1},\vec{y}_{n+1})\vec{\lambda}_{n+1},\;
n=N-2,\dots, 0  \label{eq:aBDF_main}
\end{align}
\ese
with input direction $J^\prime(\vec{y}_N)^\T$ and the convention $\al{i}{n}=0$
for $i>k_n$, $k_{\max}=\max_n \{k_n\}$ (see also \cite{Sandu2008}). This
scheme forms together with \eqref{eq:BDF} the optimality conditions of the
nonlinear program (NLP)
\begin{align}\label{eq:NLP}
\min_{\vec{y}} \enspace & J(\vec{y}_N)\enspace \operatorname{s.t.}\enspace
\text{ \eqref{eq:BDF}}
\end{align}
with $\vec{y}^\T:=\left[\begin{array}{cccc} \vec{y}^\T_0
& \vec{y}^\T_1 & \cdots & \vec{y}^\T_N \end{array}\right]$. This NLP is a
discretization of the constrained variational problem \eqref{eq:OCP}.

The discrete adjoints given by \eqref{eq:aBDF} are the exact derivatives of
the nominal integration scheme \eqref{eq:BDF} (beside round-off errors).
Furthermore, for a BDF scheme with constant order $k$, the discrete adjoint
$\vec{\lambda}_1$ converges with the same order $k$ to the value
$\vec{\lambda}(\ts)$ of the adjoint solution of \eqref{eq:aIVP}, cf.
\cite{Bock1987,Sandu2008}.\\
The discrete adjoints are generally inconsistent approximations to the solution
of \eqref{eq:aIVP} around a nominal approximation passing through
$\{\vec{y}_n\}_{n=0}^N$, see Figure \ref{fig:subf_adapt_discrAdj}. In the case
of constant order $k$ and constant stepsizes $h$, the discrete adjoints coming
from
the adjoint initialization and adjoint termination are
inconsistent as well, whereas the main part, i.e.
formula \eqref{eq:aBDF_main} with $n=N-k,\dots,k$, gives consistent
approximations of order $k$, see Figure \ref{fig:subf_BDF2_discrAdj}.

\begin{figure}
\centering
\subfigure[Non-adaptive BDF method with constant order and constant stepsizes.
\label{fig:subf_BDF2_discrAdj}]{\includegraphics[
width=0.47\textwidth]{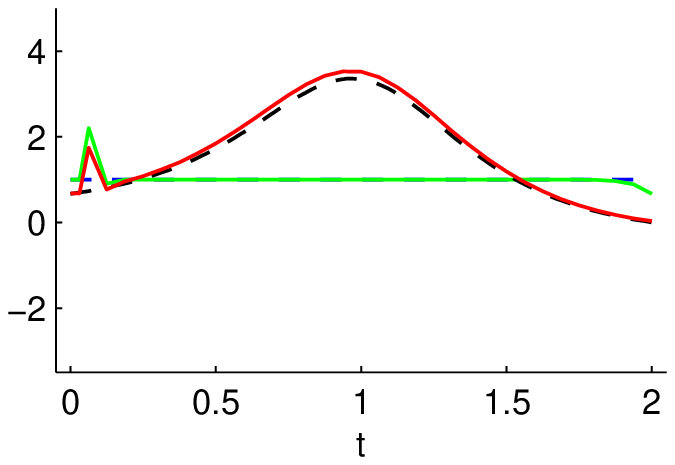}}
\quad
\subfigure[Adaptive BDF method with variable order and variable stepsize.
\label{fig:subf_adapt_discrAdj}]{\includegraphics[width =
0.47\textwidth]{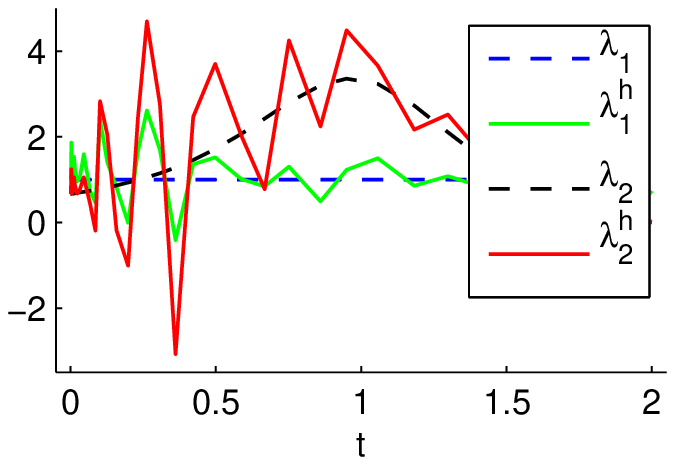}} 
\caption{Comparison of discrete adjoints $\vec{\lambda}^h =
[\lambda_1^h,\lambda_2^h]^\T$ and analytic solution
$\vec{\lambda}=[\lambda_1,\lambda_2]^\T$ of the Hilbert space adjoint
differential equation on the Catenary test case (see
Section \ref{sec:numres}).}
\label{fig:discrAdj}
\end{figure}

Due to the inconsistency of the discrete adjoint scheme \eqref{eq:aBDF} with the
adjoint differential equation \eqref{eq:aIVP} discretization and optimization
of \eqref{eq:OCP} do not commute in the commonly used Hilbert space setting.
This gives rise to the question for a new functional-analytic setting that is
suitable for multistep methods. The next sections are devoted to the
development of this setting.

\section{Solution of the constrained variational problem in a Banach space
setting}
\label{sec:B_CVP}

As seen in the previous section, the Hilbert space setting of Section
\ref{sec:Hadj} is not suitable to analyze multistep methods and their discrete
adjoints. Here, we propose to embed the constrained variational problem
\eqref{eq:OCP} into a Banach space setting and show the well-posedness of the
corresponding infinite-dimensional optimality conditions.

\subsection{General considerations}
\label{sec:gCons}

\paragraph{Duality pairing}
According to Section \ref{sec:EUDnomsol}, the residual $\vec{\rho}(\vec{y})$ of
\eqref{eq:IVP_ode} is an element of the space $C^0[\ts,\tf]^d$. Thus, we focus
on the duality pairing between the Banach space $C^0[\ts,\tf]^d$ and its dual.
The Riesz Representation Theorem \cite[Ch.5\S5.5]{Luenberger1969} 
states that for every continuous linear functional $\mathfrak{L}$ on
$C^0[\ts,\tf]$ exists a unique $\Psi \in \NBV[\ts,\tf]$ such that
\begin{align}\label{eq:int_g_Psi}
\mathfrak{L} [g] = \left\langle \Psi,g \right\rangle
_{\NBV[\ts,\tf],C^0[\ts,\tf]} = \int_\ts^\tf g(t) \td \Psi(t),
\end{align}
where the integral is the Riemann-Stieltjes integral
\cite[Ch.VIII\S6]{Natanson1977}. The Banach space $\NBV[\ts,\tf]$ consists of
all normalized functions of bounded variation on $[\ts,\tf]$ that are zero in
$\ts$ and continuous from the right on $(\ts,\tf)$. It is equipped with the
total variation norm
\begin{align}
\nrm{\Psi}{\NBV[\ts,\tf]} = \sup \sum_{i=1}^{m}
\abs{\Psi(t_i)-\Psi(t_{i-1})}\nonumber
\end{align}
where the supremum is taken over all partitions $\ts=t_0<\dots<t_m=\tf$ of
$[\ts,\tf]$. According to the Riesz Representation Theorem, for each $\Psi$ the
value of the total variation norm coincides with the value of the dual norm
given by
\begin{align}
\nrm{\Psi}{\NBV[\ts,\tf]}= \max_{\nrm{g}{C^0[\ts,\tf]}=1} \abs{\left\langle
\Psi,g \right\rangle _{\NBV[\ts,\tf],C^0[\ts,\tf]}}. \nonumber
\end{align}
Hence, we will always use the norm that is better suited in the particular
situation. The dual of the finite Cartesian product $C^0[\ts,\tf]^d$ is the
finite Cartesian product $\NBV[\ts,\tf]^d$ of the duals %
with duality pairing
\begin{align}
\left\langle \vec{\Psi},\vec{g} \right\rangle _{{\NBV}^d,\left( C^0\right)^d}
= \sum_{i=1}^d \left\langle \Psi_i,g_i \right\rangle_{\NBV,C^0} 
= \sum_{i=1}^d  \int_\ts^\tf g_i(t) \td \Psi_i(t) =:
\int_\ts^\tf \vec{g}(t) \td \vec{\Psi}(t)\nonumber
\end{align}
and dual norm $ \nrm{\Psi}{\NBV[\ts,\tf]^d} = \max_{1\leq i\leq d}
\nrm{\Psi_i}{\NBV[\ts,\tf]}$, see \cite[Ch.II\S12.1]{Wloka1971}.

\paragraph{Variational formulation of the initial value problem}
The variational formulation of \eqref{eq:IVP} on the described Banach spaces
reads: Find $\vec{y}\in C^1[\ts,\tf]^d$ with $\vec{y}(\ts)=\vec{y}_\s$ such that
\begin{align} \label{eq:varf}
\int_\ts^\tf \dot{\vec{y}}(t)-\vec{f}(t,\vec{y}(t)) \, \td\vec{\Gamma}(t) = 0
\quad \forall \vec{\Gamma} \in \NBV[\ts,\tf]^d.
\end{align}
This problem possesses at least one solution which is the strong
solution given by \eqref{eq:IVP}. The uniqueness follows from the fact that for 
continuous functions $g \in C^0[\ts,\tf]$ it holds
\begin{align}
\int_\ts^\tf g(t) \, \td\Psi (t)=0 \quad \forall \Psi \in \NBV[\ts,\tf] \quad 
\Rightarrow  \quad g=0. \nonumber
\end{align}
Thus, both formulations \eqref{eq:IVP} and \eqref{eq:varf} give the same
solution $\vec{y}(t)$ and \eqref{eq:varf} is well-posed according to the
well-posedness of \eqref{eq:IVP}.

\subsection{Infinite-dimensional optimality conditions}

Considering the constrained variational problem \eqref{eq:OCP} on the function
space $C^1[\ts,\tf]^d$, the Lagrangian $\Lag:C^1[\ts,\tf]^d \times
\NBV[\ts,\tf]^d \times \real^d \rightarrow \real$ is given by
\begin{align}\label{eq:BLag}
 \Lag(\vec{y},\vec{\Lambda},\vec{l}) := J(\vec{y}(\tf)) - \int_\ts^\tf
\dot{\vec{y}}(t)-\vec{f}(t,\vec{y}(t)) \, \td\vec{\Lambda}(t) -
\vec{l}^\T \left[\vec{y}(\ts)-\vec{y}_\s \right]
\end{align}
where the Lagrange multipliers $\vec{l}$ and $\vec{\Lambda}$ lie in the
corresponding dual spaces $\real^d$ and $\NBV[\ts,\tf]^d$. The Lagrangian is
based on the variational formulation \eqref{eq:varf} and includes the initial
condition using an additional Lagrange multiplier. We first state the central
theorem of this section which describes the stationary point of $\Lag$ and defer
the proof for the end of the section.

\begin{theorem}\label{theo:Boptcond}
The optimality conditions of the constrained variational problem \eqref{eq:OCP}
on $C^1[\ts,\tf]^d$, i.e.
\bse\label{eq:Boptcond}
\begin{align}
J^\prime(\vec{y}(\tf)) \vec{w}(\tf) - \int_\ts^\tf
\dot{\vec{w}}(t)-\vec{f}_{\vec{y}}(t,\vec{y}(t))\vec{w}(t) \,
\td\vec{\Lambda}(t) - \vec{l}^\T \vec{w}(\ts) &=0,\label{eq:Boptcond_adj}\\
-\int_\ts^\tf \dot{\vec{y}}(t)-\vec{f}(t,\vec{y}(t)) \,
\td \vec{\Gamma}(t) &=0,\label{eq:Boptcond_nom} \\
- \vec{r}^\T \left[\vec{y}(\ts)-\vec{y}_\s \right]&=0,\label{eq:Boptcond_ic}\\
\forall (\vec{w},\vec{\Gamma},\vec{r}) \in C^1[\ts,\tf]^d \times
\NBV[\ts,\tf]^d \times \real^d, \nonumber
\end{align}
\ese
possess a unique solution $(\vec{y},\vec{\Lambda},\vec{l})$ in
$C^1[\ts,\tf]^d \times \NBV[\ts,\tf]^d \times \real^d$. Moreover, $\vec{y}(t)$
is the solution of \eqref{eq:IVP}, and $\vec{l}$ and $\vec{\Lambda}(t)$ are
given in terms of the adjoint solution $\vec{\lambda}(t)$
of \eqref{eq:aIVP}
\begin{align}\label{eq:def_adjsol}
\vec{l}&=\vec{\lambda}(\ts),\quad
\vec{\Lambda}(t)=\int_{\ts}^t \vec{\lambda}(\tau) \td \tau ,
\end{align}
with componentwise integration.
\end{theorem}

The necessary optimality condition for a stationary point
$(\vec{y},\vec{\Lambda},\vec{l})$ of the Lagrangian \eqref{eq:BLag} is given by
\begin{align}
\left(\begin{array}{c}
\Lag_{\vec{y}}(\vec{y},\vec{\Lambda},\vec{l})(\vec{w})\\ 
\Lag_{\vec{\Lambda}}(\vec{y},\vec{\Lambda},\vec{l})(\vec{\Gamma}) \\
\Lag_{\vec{l}}(\vec{y},\vec{\Lambda},\vec{l})(\vec{r})
\end{array}\right)
= \left(\begin{array}{c}0\\0\\0\end{array}\right),\quad 
\forall \vec{w} \in C^1[\ts,\tf]^d,\enskip \vec{\Gamma} \in
\NBV[\ts,\tf]^d,\enskip \vec{r}\in \real^d \nonumber
\end{align}
which is exactly \eqref{eq:Boptcond}. As equations
\eqref{eq:Boptcond_nom}-\eqref{eq:Boptcond_ic} are already given by
\eqref{eq:varf} and discussed over there, we now focus on equation
\eqref{eq:Boptcond_adj} of the
optimality conditions. Provided that $\vec{y}(t)$ is known, the adjoint problem
in variational formulation reads: Find
$\left(\vec{\Lambda},\vec{l}\right)\in \NBV[\ts,\tf]^d\times \real^d$ such that
\eqref{eq:Boptcond_adj} holds for all $\vec{w} \in C^1[\ts,\tf]^d$.

\begin{lemma}\label{lem:exsol_varf_adj}
For the solution $\vec{y}(t)$ of \eqref{eq:Boptcond_nom}-\eqref{eq:Boptcond_ic},
a corresponding adjoint solution $\left(\vec{\Lambda},\vec{l}\right)\in
\NBV[\ts,\tf]^d\times \real^d$ of \eqref{eq:Boptcond_adj} is provided by
\eqref{eq:def_adjsol}.
\end{lemma}

\begin{proof}
Recall that the adjoint differential equation \eqref{eq:aIVP} has a unique
solution $\vec{\lambda} \in C^1[\ts,\tf]^d$ (cf. Section \ref{sec:HLag}).
Multiplying the transposed of \eqref{eq:aIVP_ode} from the right by
any $\vec{w}\in C^1[\ts,\tf]^d$, integrating over $[\ts,\tf]$ and adding the
transposed of \eqref{eq:aIVP_ic} multiplied by $\vec{w}(\tf)$ yields
\begin{align} \label{eq:stradjint}
\int^{\tf}_{\ts}
\left[\dot{\vec{\lambda}}(t)+\vec{f}_{\vec{y}}^\T(t,\vec{y}(t))\vec{\lambda}(t)
\right]^\T \vec{w}(t) \td t - \left[\vec{\lambda}(\tf)-
J^\prime(\vec{y}(\tf))^\T \right]^\T \vec{w}(\tf) =0 .
\end{align}
Integration by parts gives for all $\vec{w}\in C^1[\ts,\tf]^d$
\begin{align}
\int^{\tf}_{\ts}\vec{\lambda}^\T(t) \left[
\dot{\vec{w}}(t)-\vec{f}_{\vec{y}}(t,\vec{y}(t))\vec{w}(t)\right]\td t -
\vec{\lambda}^\T(\ts)\vec{w}(\ts) +
J^\prime(\vec{y}(\tf))\vec{w}(\tf)=0.\nonumber
\end{align}
Consequently, \eqref{eq:def_adjsol} provides a solution $(\vec{\Lambda},\vec{l})
\in \NBV[\ts,\tf]^d\times \real^d$ of \eqref{eq:Boptcond_adj}, since the
indefinite integral $\Lambda_i(t)=\int_\ts^t\lambda_i(\tau)\td \tau$ is a
normalized function of bounded variation \cite[Sec.32]{Kolmogorov1970} 
and it holds $\int_\ts^\tf g(t)\td
\Lambda_i(t)= \int_\ts^\tf \Lambda^\prime_i(t)g(t)\td t=\int_\ts^\tf
\lambda_i(t)g(t)\td t$, cf. \cite[Ch.VIII\S6]{Natanson1977}.
\end{proof}

The next lemma proves the uniqueness of the weak adjoint solution.

\begin{lemma}\label{lem:usol_varf_adj}
For the solution $\vec{y}(t)$ of \eqref{eq:Boptcond_nom}-\eqref{eq:Boptcond_ic},
the corresponding adjoint solution $(\vec{\Lambda},\vec{l}) \in \NBV[\ts,\tf]^d
\times \real^d$ of \eqref{eq:Boptcond_adj} is unique.
\end{lemma}

\begin{proof}
Equation \eqref{eq:Boptcond_adj} is equivalent to
\begin{align}
\underbrace{\int_\ts^\tf
\dot{\vec{w}}(t)-\vec{f}_{\vec{y}}(t,\vec{y}(t))\vec{w}(t) \,
\td\vec{\Lambda}(t)+ \vec{l}^\T
\vec{w}(\ts)}_{=:\op{A}(\vec{\Lambda},\vec{l})(\vec{w})} =
\underbrace{J^\prime(\vec{y}(\tf)) \vec{w}(\tf)}_{=: B(\vec{w})} \quad \forall
\vec{w}\in C^1[\ts,\tf]^d  \nonumber
\end{align}
where $B$ and $\op{A}(\vec{\Lambda},\vec{l})$ are linear functionals on
$C^1[\ts,\tf]^d$ and $\op{A}:\NBV[\ts,\tf]^d \times \real^d \rightarrow
\left( C^1[\ts,\tf]^d\right)^\prime$ is linear in $(\vec{\Lambda},\vec{l})$. We
have to show that $\mN(\op{A}) = \left\lbrace (\vec{0},\vec{0})\right\rbrace$,
where the nullspace of $\op{A}$ is given by
\begin{align}
\mN(\op{A}) = \left\lbrace (\vec{\Lambda},\vec{l}) \in \NBV[\ts,\tf]^d \times
\real^d: \;
\op{A}(\vec{\Lambda},\vec{l})(\vec{w})=0 \quad \forall \vec{w}\in
C^1[\ts,\tf]^d \right\rbrace .\nonumber
\end{align}
Due to Section \ref{sec:EUDnomsol}, for every initial value $\vec{w}_1(\ts)\in
\real^d$ there exists a function $\vec{w}_1\in C^1[\ts,\tf]^d$ that satisfies
\eqref{eq:fIVP_ode}. Inserting $\vec{w}_1$ in $\op{A}(\vec{\Lambda},\vec{l})$
then gives
\begin{align}
\op{A}(\vec{\Lambda},\vec{l})(\vec{w}_1)
= \int_\ts^\tf \vec{0} \,\td\vec{\Lambda}(t)+ \vec{l}^\T \vec{w}_1(\ts) 
= 0 +\vec{l}^\T \vec{w}_1(\ts) .\nonumber
\end{align}
Thus, $\vec{l}$ has to vanish in
order to ensure $\op{A}(\vec{\Lambda},\vec{l})(\vec{w})=0 \; \forall \vec{w}\in
C^1[\ts,\tf]^d$. Now, we search for functions $\vec{\Lambda} \in
\NBV[\ts,\tf]^d$ with
\begin{align}
\op{A}(\vec{\Lambda},\vec{0})(\vec{w})=\int_\ts^\tf
\dot{\vec{w}}(t)-\vec{f}_{\vec{y}}(t,\vec{y}(t))\vec{w}(t) \,
\td\vec{\Lambda}(t) =0 \quad \forall \vec{w}\in C^1[\ts,\tf]^d. \nonumber
\end{align}
With $\vec{g}(t):= \dot{\vec{w}}(t)-\vec{f}_{\vec{y}}(t,\vec{y}(t))\vec{w}(t)$,
it is the same to vary either $\vec{w}\in C^1[\ts,\tf]^d$ or $\vec{g} \in
C^0[\ts,\tf]^d$, since the inhomogeneous ODE possesses a unique solution
$\vec{w}(t)$ for every $\vec{g}(t)$. According to the uniqueness of $\Psi$ in
\eqref{eq:int_g_Psi} it holds
\begin{align}
\int_\ts^\tf \vec{g}(t) \, \td\vec{\Lambda}(t) =0 \quad \forall \vec{g}\in
C^0[\ts,\tf]^d \quad \Rightarrow \quad \vec{\Lambda} = \vec{0}.\nonumber
\end{align}
Consequently, $\mN(\op{A}) = \left\lbrace (\vec{0},\vec{0})\right\rbrace$ which
proves the uniqueness of the solution of \eqref{eq:Boptcond_adj}.
\end{proof}

With this knowledge at hand we can now come to the proof of Theorem
\ref{theo:Boptcond}.

\begin{proof}[of Theorem \ref{theo:Boptcond}]
As seen in Section \ref{sec:gCons}, the equations 
\eqref{eq:Boptcond_nom}-\eqref{eq:Boptcond_ic} have the same unique solution
$\vec{y}(t)$ as \eqref{eq:IVP} which implies their well-posedness. According
to Lemma \ref{lem:exsol_varf_adj}, a solution of \eqref{eq:Boptcond_adj} is
provided by \eqref{eq:def_adjsol}. Furthermore, it is the only solution of
\eqref{eq:Boptcond_adj} according to Lemma \ref{lem:usol_varf_adj}. Since
$\vec{\lambda}(t)$ depends continuously on $J^\prime(\vec{y}(\tf))^\T$ (cf.
Section \ref{sec:HLag}) this still holds for $\vec{\Lambda}(t)$ and $\vec{l}$.
Thus, \eqref{eq:Boptcond_adj} together with
\eqref{eq:Boptcond_nom}-\eqref{eq:Boptcond_ic} is well-posed.
\end{proof}

With the concept of weak solutions from partial differential equations
(see e.g. \cite{Johnson1987}), the triple $(\vec{y},\vec{\Lambda},\vec{l})$ is a
weak solution of \eqref{eq:IVP} and \eqref{eq:aIVP}, since it solves the
variational formulation \eqref{eq:Boptcond} of \eqref{eq:IVP} and
\eqref{eq:aIVP}. Thus, we will call $\vec{\Lambda}$ a \emph{weak adjoint
solution} of \eqref{eq:aIVP} or shortly \emph{weak adjoint}. Note that for the
nominal solution, the weak solution $\vec{y}$ defined by
\eqref{eq:Boptcond_ic}-\eqref{eq:Boptcond_nom} is directly the classical
solution of \eqref{eq:IVP}. Whereas for the adjoint, the weak solution
$\vec{\Lambda}$ is sufficiently regular such that a classical solution of
\eqref{eq:aIVP} is provided by $\vec{\Lambda}^\prime = \vec{\lambda}$.

\subsection{Extension of the infinite-dimensional optimality conditions}

As seen in Section \ref{sec:BDF} the approximations to the solution
of \eqref{eq:IVP} obtained from BDF methods are not continuously differentiable
on the whole interval $[\ts,\tf]$ but rather continuous and piecewise
continuously differentiable. To capture this case, an appropriate extension of
the trial space $C^1[\ts,\tf]^d$ is required. To this end, we employ a time grid
$\ts=t_0 < t_1 <\dots <t_N=\tf$ and a partition \label{partition} of $[\ts,\tf]$
using subintervals $I_n=(t_n,t_{n+1}]$ of length $h_n=t_{n+1}-t_n$ such that
$[\ts,\tf]=\{\ts\} \cup I_0 \cup\dots \cup I_{N-1}$. Choosing the trial space
as 
\begin{align} \label{eq:defspaceY}
Y[\ts,\tf]^d:=\left\lbrace \vec{y} \in C^0[\ts,\tf]^d:
\vec{y}\big\vert_{I_n} \in C_b^1(I_n)^d \right\rbrace,
\end{align}
where $C_b^1(I_n)$ is the space of all continuously differentiable and
bounded functions with bounded derivative \cite[Ch.1]{Adams2003}, %
the extended Lagrangian $\eLag:Y[\ts,\tf]^d\times \NBV[\ts,\tf]^d \times
\real^d \rightarrow \real$ of \eqref{eq:OCP} solved on the function space
$Y[\ts,\tf]^d$
is
\begin{align}
\eLag(\vec{y},\vec{\Lambda},\vec{l}) := J(\vec{y}(\tf)) -
\sum_{n=0}^{N-1}
\int_{I_n} \dot{\vec{y}}(t)-\vec{f}(t,\vec{y}(t)) \, \td\vec{\Lambda}(t) -
\vec{l}^\T \left[\vec{y}(\ts)-\vec{y}_\s \right]. \nonumber
\end{align}
The Lagrangian $\eLag$ is based on the extension $\hat{\mathfrak{L}}$ of
the linear functional $\mathfrak{L}$ given by \eqref{eq:int_g_Psi} from
$C^0[\ts,\tf]$ to $Y[\ts,\tf]$. The existence of $\hat{\mathfrak{L}}$ is
guaranteed due to \cite[p. 89]{Wloka1971}. We define the extended
Riemann-Stieltjes integral on $I_n=(t_n,t_{n+1}]$ using the partition
$t_n<\tau_0 < \tau_1 <\dots < \tau_m=t_{n+1}$ and the convention that
$\theta_k=\tau_{k-1}\in[\tau_{k-1},\tau_k]$ for $k=1,\dots,m$ by
\begin{align} \label{eq:extRSint}
 \int_{(t_n,t_{n+1}]} g(t) \td \Psi(t) = \lim_{m\rightarrow \infty}
\sum_{k=1}^m g(\tau_k)[\Psi(\tau_k)-\Psi(\tau_{k-1})]
\end{align}
such that
\begin{align}
\hat{\mathfrak{L}}[g] = \sum_{n=0}^{N-1} \int_{I_n}g(t)\td \Psi(t).\nonumber
\end{align}
This extension $\hat{\mathfrak{L}}$ restricted to the continuous functions
$g\in C^0[\ts,\tf]$ coincides with $\mathfrak{L}$. Thus, the same holds for the 
Lagrangian $\eLag$. Furthermore, if $g\in C^0[t_n,t_{n+1}]$ then
$\int_{t_n}^{t_{n+1}}g(t)\td\Psi(t)= \int_{I_n} g(t) \td \Psi(t)$.\\
With these definitions at hand, we first state the main result of the section.

\begin{theorem}\label{theo:exBoptcond}
The optimality conditions of the constrained variational problem \eqref{eq:OCP}
on $Y[\ts,\tf]^d$, i.e.
\bse\label{eq:exBoptcond}
\begin{align}
J^\prime(\vec{y}(\tf)) \vec{w}(\tf) - \sum_{n=0}^{N-1}
\int_{I_n}\dot{\vec{w}}(t)-\vec{f}_{\vec{y}}(t,\vec{y}(t))\vec{w}(t) \,
\td\vec{\Lambda}(t) - \vec{l}^\T \vec{w}(\ts) &=0,\label{eq:exBoptcond_adj}\\
-\sum_{n=0}^{N-1}
\int_{I_n} \dot{\vec{y}}(t)-\vec{f}(t,\vec{y}(t)) \,
\td \vec{\Gamma}(t) &=0,\label{eq:exBoptcond_nom} \\
- \vec{r}^\T \left[\vec{y}(\ts)-\vec{y}_\s \right]&=0,\label{eq:exBoptcond_ic}\\
\forall (\vec{w},\vec{\Gamma},\vec{r}) \in Y[\ts,\tf]^d \times
\NBV[\ts,\tf]^d \times \real^d, \nonumber
\end{align}
\ese
possess a unique solution $(\vec{y},\vec{\Lambda},\vec{l})$ in $Y[\ts,\tf]^d
\times \NBV[\ts,\tf]^d \times \real^d$ that coincides
with the solution of \eqref{eq:Boptcond}.
\end{theorem}

We start with considering the nominal equations
\eqref{eq:exBoptcond_ic}-\eqref{eq:exBoptcond_nom}.

\begin{lemma}\label{lem:exsol_extform}
The solution $\vec{y}(t)$ of \eqref{eq:Boptcond_ic}-\eqref{eq:Boptcond_nom}
solves the extended variational formulation
\eqref{eq:exBoptcond_ic}-\eqref{eq:exBoptcond_nom}.
\end{lemma}

\begin{proof}
Let $\vec{y}(t)$ be the solution of
\eqref{eq:Boptcond_ic}-\eqref{eq:Boptcond_nom}. From $C^1[\ts,\tf]^d
\subset Y[\ts,\tf]^d$ follows that $\vec{y}\in Y[\ts,\tf]^d$. Since the
integral
$\int_\ts^\tf g_i(t)\td\Gamma_i(t)$ with $g_i(t):= \dot{y}_i(t)-f_i(t,y(t)) $
exists, also the integrals over the subintervals $\int_{t_n}^{t_{n+1}}
g_i(t) \td\Gamma_i(t)$ exist and it holds \cite[Ch.VIII\S6]{Natanson1977}
\begin{align}
\int_\ts^\tf g_i(t) \td\Gamma_i(t) = \sum_{n=0}^{N-1} 
\int_{t_n}^{t_{n+1}} g_i(t) \td\Gamma_i(t) = \sum_{n=0}^{N-1} 
\int_{I_n} g_i(t) \td\Gamma_i(t)\nonumber
\end{align}
where the second equality is due to the extension \eqref{eq:extRSint} of the
Riemann-Stieltjes integral, $i=1,\dots,d$. Thus, equation
\eqref{eq:Boptcond_nom} becomes $\forall
\vec{\Gamma}\in \NBV[\ts,\tf]^d$
\begin{align}
0 &= \int_\ts^\tf \dot{\vec{y}}(t)-\vec{f}(t,\vec{y}(t)) \, \td\vec{\Gamma}(t)
= \sum_{n=0}^{N-1} \int_{I_n} \dot{\vec{y}}(t)-\vec{f}(t,\vec{y}(t)) \,
\td\vec{\Gamma}(t) \nonumber
\end{align}
which coincides with \eqref{eq:exBoptcond_nom}.
\end{proof}

\begin{lemma}\label{lem:usol_extform}
The extended variational formulation
\eqref{eq:exBoptcond_nom}-\eqref{eq:exBoptcond_ic} possesses a unique solution
$\vec{y}(t)$.
\end{lemma}

\begin{proof}
Let $\vec{y}(t)$ be a solution of
\eqref{eq:exBoptcond_nom}-\eqref{eq:exBoptcond_ic}. The space $\NBV[\ts,\tf]^d$
contains, in particular, the functions that vanish everywhere except on
$(t_n,t_{n+1})$. Thus, a necessary condition for $\vec{y}(t)$ being a solution
of \eqref{eq:exBoptcond_nom}-\eqref{eq:exBoptcond_ic} is that each addend has
to vanish, i.e. $\int_{I_n} \dot{\vec{y}}(t)-\vec{f}(t,\vec{y}(t))
\td\vec{\Gamma}(t)=0 \enskip \forall \vec{\Gamma} \in \NBV(I_n)^d$ with
$\vec{\Gamma}(t_{n+1})=\vec{0}$. The fundamental theorem of
variational calculus yields $\dot{\vec{y}}(t)-\vec{f}(t,\vec{y}(t))=\vec{0}$ on
$(t_n,t_{n+1})$ for all $n=0,\dots,N-1$. On the other hand, $\NBV[\ts,\tf]^d$
contains also the constant
functions having a single jump in $t_n$. They give the necessary conditions
$\dot{\vec{y}}(t_n)-\vec{f}(t_n,\vec{y}(t_n))=\vec{0}$ for $n=1,\dots,N$. Since
$\vec{f}(t,\vec{y})$ is continuous in both variables and $\vec{y}\in
C^0[\ts,\tf]^d$, $\vec{y}(t)$ is necessarily continuously differentiable on
$[\ts,\tf]$. Thus, every solution of
\eqref{eq:exBoptcond_nom}-\eqref{eq:exBoptcond_ic} satisfies
\eqref{eq:Boptcond_nom}-\eqref{eq:Boptcond_ic} which possesses a unique
solution.
\end{proof}

As conclusion of this lemma, the dependency of the solution of the extended
variational formulation \eqref{eq:exBoptcond_nom}-\eqref{eq:exBoptcond_ic} on
the input data is continuous and thus the problem is well-posed.\\

Now, we focus on the adjoint problem in extended variational formulation which
is for a given $\vec{y}(t)$: Find $(\vec{\Lambda},\vec{l}) \in
\NBV[\ts,\tf]^d \times \real^d$ such that \eqref{eq:exBoptcond_adj} holds for
all $\vec{w}\in Y[\ts,\tf]^d$.

\begin{lemma}\label{lem:exsol_extvarf_adj}
For the solution $\vec{y}(t)$ of
\eqref{eq:exBoptcond_nom}-\eqref{eq:exBoptcond_ic}, the corresponding adjoint
solution $\left(\vec{\Lambda},\vec{l}\right)\in \NBV[\ts,\tf]^d\times
\real^d$ of \eqref{eq:exBoptcond_adj} is provided by \eqref{eq:def_adjsol}.
\end{lemma}

\begin{proof}
We proceed in the same way as in the proof of Lemma \ref{lem:exsol_varf_adj},
but choose $\vec{w} \in Y[\ts,\tf]^d$ for the multiplication and split the
integral in \eqref{eq:stradjint} using the subintervals $I_n$ (same arguments as
in the proof of Lemma \ref{lem:exsol_extform}). Integration by
parts of all integrals yields the equivalent equation
\begin{align}
-\vec{\lambda}^\T(\ts) \vec{w}(\ts) -\sum_{n=0}^{N-1}
\int_{I_n}\vec{\lambda}^\T(t) \left[
\dot{\vec{w}}(t)-\vec{f}_{\vec{y}}(t,\vec{y}(t))\vec{w}(t)\right]\td t 
+ J^\prime(\vec{y}(\tf)) \vec{w}(\tf)=0. \nonumber
\end{align}
Thus, the choice \eqref{eq:def_adjsol} provides a solution of
\eqref{eq:exBoptcond_adj}.
\end{proof}

\begin{lemma}\label{lem:usol_extvarf_adj}
For the solution $\vec{y}(t)$ of
\eqref{eq:exBoptcond_nom}-\eqref{eq:exBoptcond_ic}, the corresponding adjoint
solution $(\vec{\Lambda},\vec{l}) \in \NBV[\ts,\tf]^d \times \real^d$ of
\eqref{eq:exBoptcond_adj} is unique.
\end{lemma}

\begin{proof} 
We follow mainly the proof of Lemma \ref{lem:usol_varf_adj}. Equation
\eqref{eq:exBoptcond_adj} is equivalent to
\begin{align}
\underbrace{\sum_{n=0}^{N-1} \int_{I_n}
\dot{\vec{w}}(t)-\vec{f}_{\vec{y}}(t,\vec{y}(t))\vec{w
}(t) \, d\vec{\Lambda}(t)+
\vec{l}^\T \vec{w}(\ts)}_{=:\hat{\op{A}}(\vec{\Lambda},\vec{l})(\vec{w})} =
\underbrace{J^\prime(\vec{y}(\tf)) \vec{w}(\tf)}_{=B(\vec{w})} \; \forall
\vec{w}\in Y[\ts,\tf]^d \nonumber
\end{align}
where $\hat{\op{A}}(\vec{\Lambda},\vec{l})$ is also a linear functional on
$Y[\ts,\tf]^d$ and $\hat{\op{A}}:\NBV[\ts,\tf]^d \times \real^d \rightarrow
\left( Y[\ts,\tf]^d\right)^\prime$ is linear in $(\vec{\Lambda},\vec{l})$. We
show again that $\mN(\hat{\op{A}}) = \left\lbrace
(\vec{0},\vec{0})\right\rbrace$. Since $C^1[\ts,\tf]^d \subset Y[\ts,\tf]^d$,
$\vec{l}$ has to vanish due to the same arguments as used in the proof of Lemma
\ref{lem:usol_varf_adj}. Thus, the following equation
\begin{align} \label{eq:extd_ker_Psi}
\hat{\op{A}}(\vec{\Lambda},\vec{0})(\vec{w})=\sum_{n=0}^{N-1} \int_{I_n}
\dot{\vec{w}}(t)-\vec{f}_{\vec{y}}(t,\vec{y}(t))\vec{w
}(t) \, \td\vec{\Lambda}(t)
=0  \quad \forall \vec{w}\in Y[\ts,\tf]^d \nonumber
\end{align}
has to be satisfied also for $\vec{w}\in C^1[\ts,\tf]^d \subset Y[\ts,\tf]^d$,
i.e. with
$\vec{g}(t):=\dot{\vec{w}}(t)-\vec{f}_{\vec{y}}(t,\vec{y}(t))\vec{w}(t)$
it becomes
\begin{align}
\sum_{n=0}^{N-1} \int_{I_n} \vec{g}(t) \, \td\vec{\Lambda}(t) =0  \quad \forall
\vec{g}\in C^0[\ts,\tf]^d. \nonumber
\end{align}
Furthermore, as $\vec{g}(t)$ is continuous the integral $\int_\ts^\tf
\vec{g}(t)\td \vec{\Lambda}(t)$ exists and coincides with the sum of the
integrals over the subintervals (same arguments as
in the proof of Lemma \ref{lem:exsol_extform})
and the proof
can be finished in the same way as that of Lemma \ref{lem:usol_varf_adj}. 
\end{proof}

With all this at hand we are able to prove Theorem \ref{theo:exBoptcond}.

\begin{proof}[of Theorem \ref{theo:exBoptcond}]
Lemma \ref{lem:exsol_extform} and \ref{lem:usol_extform} prove the existence of
a unique solution of \eqref{eq:exBoptcond_nom}-\eqref{eq:exBoptcond_ic}
coinciding with the solution of \eqref{eq:Boptcond_nom}-\eqref{eq:Boptcond_ic}.
For this solution, equation \eqref{eq:exBoptcond_adj} has a unique solution
given by \eqref{eq:def_adjsol} due to Lemma \ref{lem:exsol_extvarf_adj}  and
\ref{lem:usol_extvarf_adj}.
\end{proof}

\section{Petrov-Galerkin discretization of the extended optimality conditions}
\label{sec:PGappr}

In order to solve the infinite-dimensional optimality conditions
\eqref{eq:exBoptcond} numerically, the infinite-dimensional function spaces have
to be approximated by finite-dimensional subspaces, the finite element
spaces. This so-called \textit{Petrov-Galerkin approximation} transfers the
infinite-dimensional conditions into a finite-dimensional system of
equations which can be solved on a computer. The first part of the section
focuses on the finite-dimensional subspace, and the second part is devoted to
the resulting system of equations.

\subsection{Finite element spaces}
\label{sec:FEspaces}

This section deals with the discretization of the function spaces
$Y[\ts,\tf]^d$ and $\NBV[\ts,\tf]^d$ by choosing appropriate sets of basis
functions.

\paragraph{Trial space}
To discretize the trial space $Y[\ts,\tf]^d$ we use piecewise polynomials of
order $k_n$ on the subinterval $I_n$
\begin{align}
Y_{\mP}[\ts,\tf]^d:=\left\lbrace \vec{y} \in C^0[\ts,\tf]^d:
\vec{y}\big\vert_{I_n} \in \mP^{(k_n)}(I_n)^d \right\rbrace.
\end{align}
We choose local basis functions $\phi_n$ that are composed of the fundamental
Lagrangian polynomials \eqref{eq:Lagpol} restricted to the particular
subinterval. Figure \ref{fig:basisfct} shows the basis function $\phi_n \in
Y_{\mP}[\ts,\tf]^d$ with $k_0=1$, $k_n=2$ for $n>0$ and $h_n=h$ for all $n$. The
support of a single basis function depends on the orders and contains at most
seven adjacent subintervals as BDF methods are stable up to order $6$.

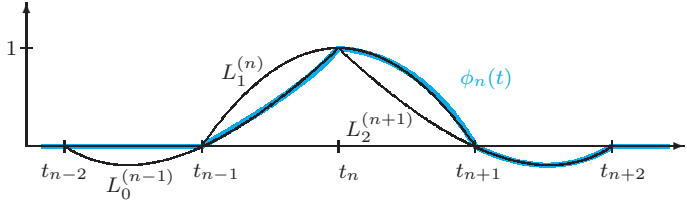
\begin{figure}
\centering
\setlength{\unitlength}{0.5mm}
\begin{picture}(173,52)(-9,3)

\linethickness{.5mm}
{\color{cyan}
\put(36,16){\line(-2,0){42} }
\bezier{300}(36,16)(63,31)(72,42)
\bezier{300}(72,42)(94,40)(108,16)
\bezier{300}(108,16)(128,6)(144,16)
\put(144,16){\line(2,0){15} }
\put(110,34){\makebox(0,0){\sf \begin{small} $\phi_n(t)$\end{small}}}
}
\thinlines

\put(-10,16){\vector(0,1){38}}
\put(-10,16){\vector(1,0){173}}

\put(0,14.5){\line(0,2){3} } 
\put(0,9){\makebox(0,0){\sf \begin{small} $t_{n-2}$ \end{small}}}
\put(20,6){\makebox(0,0){\sf \begin{small} $L_0^{(n-1)}$ \end{small}}}
\put(36,14.5){\line(0,2){3} } 
\put(39,9){\makebox(0,0){\sf \begin{small} $t_{n-1}$\end{small}}}
\put(47,36){\makebox(0,0){\sf \begin{small} $L_1^{(n)}$ \end{small}}}
\put(72,14.5){\line(0,2){3} } 
\put(74,9){\makebox(0,0){\sf \begin{small} $t_{n}$\end{small}}}
\put(83,21){\makebox(0,0){\sf \begin{small} $L_2^{(n+1)}$ \end{small}}}
\put(108,14.5){\line(0,2){3} } 
\put(108,9){\makebox(0,0){\sf \begin{small} $t_{n+1}$\end{small}}}
\put(144,14.5){\line(0,2){3} } 
\put(146,9){\makebox(0,0){\sf \begin{small} $t_{n+2}$\end{small}}}

\put(-11.5,42){\line(2,0){3} } \put(-14,42){\makebox(0,0){\sf \begin{small}
$1$ \end{small}}}

\bezier{300}(0,16)(26,-1.5)(72,42)
\bezier{300}(36,16)(72,68)(108,16)
\bezier{300}(72,42)(118,-1.5)(144,16)
\end{picture}

\caption{Basis function $\phi_n$ of $Y_{\mP}[\ts,\tf]^d$ with $k_0=1$, $k_n=2$
for $n>0$ and constant stepsizes $h_n=h$ for all $n$.}
\label{fig:basisfct}
\end{figure}

The solution $\vec{y} \in Y[\ts,\tf]^d$ is then approximated by
\begin{align}
\vec{y}(t)\approx \vec{y}^h(t):= \vec{y}_\s\phi_0(t)+ \sum_{n=1}^N \vec{y}_n
\phi_n(t)\nonumber 
\end{align}
which results in $N\cdot d$ degrees of freedom $\{\vec{y}_n\in
\real^d\}_{n=1}^N$, since the initial value $\vec{y}_0=\vec{y}_\s$ is
already fixed. To achieve locally the order $k_n>1$, former values
$\vec{y}_{n+1-k_n},\dots,\vec{y}_n$ are reused to set up the
interpolation polynomial of order $k_n$ which is afterwards restricted to
$I_n$.

\paragraph{Test space}

We approximate the test space $\NBV[\ts,\tf]^d$ using Heaviside functions
as basis functions. We choose them to be continuous from the right with
discontinuity in $t_n$. Thus, a function $\vec{\Lambda} \in \NBV[\ts,\tf]^d$ is
approximated by the linear combination of these basis functions in the form
\begin{align}\label{eq:La_h}
\vec{\Lambda}(t) \approx \vec{\Lambda}^h(t) := \sum_{n=1}^N
h_{n-1}\vec{\lambda}_n H_n(t)
\end{align}
where the $h_{n-1}$ appear for reasons which will become clear later. Note
that $\vec{\Lambda}^h$ is a step function with initial value
$\vec{\Lambda}^h(\ts)=\vec{0}$ and jumps of magnitude $h_{n-1}\vec{\lambda}_n$
at $t_n$ for $n=1,\dots,N$. Thus, it is $\vec{\Lambda}^h(t_n)=
\vec{\Lambda}^h(t_{n-1})+h_{n-1}\vec{\lambda}_n$ at the time points
and $\vec{\Lambda}^h(t)=\vec{\Lambda}^h(t_n)$ for inner points $t\in
(t_n,t_{n+1})$. We denote this space by $Z_H[\ts,\tf]^d$.

Regarding the relation \eqref{eq:def_adjsol} between the adjoint solutions
$\vec{\lambda}$ and $\vec{\Lambda}$, the classical derivative of
$\vec{\Lambda}^h$ fails to exist. But $\vec{\Lambda}^h$ is still differentiable
in a weak form such that its weak derivative is given by the Dirac measures at
$\{t_1,\dots,t_N\}$ with heights $\{h_0
\vec{\lambda}_1,\dots,h_{N-1}\vec{\lambda}_N\}$, see e.g.
\cite[Sec. 4.24]{Alt2002}.

\subsection{Finite-dimensional optimality conditions}
\label{sec:fdim_optcond}

In this section, we approximate the infinite-dimensional optimality conditions
\eqref{eq:exBoptcond} by finite-dimensional equations that result from
approximating the function spaces by the finite element spaces of Section
\ref{sec:FEspaces}. The resulting system of equations will be discussed in the
following.

\begin{theorem}\label{theo:disBoptcond}
The discretized optimality conditions, i.e.
\bse\label{eq:disBoptcond}
\begin{align}
J^\prime(\vec{y}^h(\tf)) \vec{w}^h(\tf) \hspace{7cm} \nonumber\\
 - \sum_{n=0}^{N-1}
\int_{I_n}\dot{\vec{w}}^h(t)-\vec{f}_{\vec{y}}(t,\vec{y}^h(t))\vec{w}^h(t) \,
\td\vec{\Lambda}^h(t) - [\vec{l}^h]^\T \vec{w}^h(\ts)
&=0,\label{eq:disBoptcond_adj}\\
-\sum_{n=0}^{N-1}
\int_{I_n} \dot{\vec{y}}^h(t)-\vec{f}(t,\vec{y}^h(t)) \,
\td \vec{\Gamma}^h(t) &=0,\label{eq:disBoptcond_nom} \\
- [\vec{r}^h]^\T \left[\vec{y}^h(\ts)-\vec{y}_\s
\right]&=0,\label{eq:disBoptcond_ic}\\
\forall (\vec{w}^h,\vec{\Gamma}^h,\vec{r}^h) \in Y_\mP[\ts,\tf]^d \times
Z_H[\ts,\tf]^d \times \real^d, \nonumber
\end{align}
\ese
are equivalent to the BDF scheme \eqref{eq:BDF} with prescribed stepsizes and
orders together with its discrete adjoint scheme \eqref{eq:aBDF}.
\end{theorem}

The above theorem is the main result of this section. The proof follows directly
from the two lemmas given below.

\begin{lemma} \label{lem:equiBDF}
The equations \eqref{eq:disBoptcond_nom}-\eqref{eq:disBoptcond_ic} are
equivalent to the BDF scheme \eqref{eq:BDF} with prescribed stepsizes and
orders.
\end{lemma}

\begin{proof}
We first consider one addend of \eqref{eq:disBoptcond_nom}
\label{solveint}
\begin{align}
\int_{I_n} &\dot{\vec{y}}^h(t)-\vec{f}(t,\vec{y}^h(t)) \,
\td\vec{\Gamma}^h(t)\nonumber\\
=& \left[ \vec{\Gamma}^h(t_{n+1}) - \vec{\Gamma}^h(t_n)\right]^\T
\left\lbrace \dot{\vec{y}}^h(t_{n+1})
-\vec{f}(t_{n+1},\vec{y}^h(t_{n+1}))
\right\rbrace \nonumber\\
=& \vec{\gamma}^\T_{n+1} \Bigg\lbrace \sum _{i=0}^{k_n}
\underbrace{h_n\dot{\phi}_{n+1-i}(t_{n+1}) }_{= \al{i}{n}}\vec{y}_{n+1-i}-h_n
\vec{f}(t_{n+1},\vec{y}_{n+1})\Bigg\rbrace 
\nonumber
\end{align}
where the first equality holds due to the extended
Riemann-Stieltjes integral \eqref{eq:extRSint} in vector-valued version
with coefficients $h_n\vec{\gamma}_{n+1}$ of $\vec{\Gamma}^h$ in
\eqref{eq:La_h}. The second equality uses the properties of the basis functions
$\phi_n$. Here the appearance of the $h_n$ in the coefficients of
$\vec{\Lambda}^h$ given by \eqref{eq:La_h} becomes clear. Thus,
\eqref{eq:disBoptcond_nom} can be written as a system of
equations that is nonlinear in $\{\vec{y}_n\}_{n=1}^N$ and linear in 
$\vec{\gamma}^\T:=\left[\begin{array}{cccc} \vec{\gamma}^\T_1 &
\vec{\gamma}^\T_2 & \cdots & \vec{\gamma}^\T_N \end{array}\right]\in \left(
\real^d\right)^N$
\begin{align}\label{eq:dvarf}
\vec{\gamma}^\T \left[ \left( \mat{A} \otimes \mat{I}\right)  \left(
\begin{array}{c} \vec{y}_1 \\ \vec{y}_2\\ \vdots \\ \vec{y}_N
\end{array}\right)+ 
\left( \begin{array}{c} \al{1}{0} \vec{y}_\s \\ 0\\ \vdots \\ 0
\end{array}\right) -
\left( \begin{array}{c} h_0 \vec{f}(t_1,\vec{y}_1) \\ h_1
\vec{f}(t_2,\vec{y}_2)\\ \vdots \\ h_{N-1}
\vec{f}(t_N,\vec{y}_N) \end{array}\right)\right] = 0,\; \forall
\vec{\gamma}
\end{align}
where $\mat{A}\otimes \mat{I}$ denotes the Kronecker tensor product, i.e. the
$(N\cdot d)\times( N\cdot d)$ matrix with $d\times d$ blocks $a_{ij}\mat{I}$,
and the quadratic matrix $\mat{A}$ is lower triangular with band structure
\label{p:matA}
\begin{align}
\mat{A} = \left( \begin{array}{ccccc} \al{0}{0} & 0 & 0& 0& \cdots \\ \al{1}{1}
& \al{0}{1} & 0& 0& \cdots \\ \vdots \\ \cdots & 0&\al{k_{N-1}}{N-1} &\cdots
&\al{0}{N-1}\end{array}\right).\nonumber
\end{align}
Equation \eqref{eq:dvarf} holds if and only if the term in the squared brackets
vanishes. Since $\mat{A}$ is lower triangular, each $\vec{y}_{n+1}$ is
determined
directly from $\vec{y}_\s,\vec{y}_1,\dots,\vec{y}_n$ by the $n$th equation of
the squared brackets term in \eqref{eq:dvarf} which coincides with the $n$th
step of \eqref{eq:BDF_main}. So, together with the equivalence between
\eqref{eq:BDF_ini} and \eqref{eq:disBoptcond_ic} the lemma is shown.
\end{proof}

\begin{lemma}\label{lem:equiaBDF}
For the solution $\vec{y}^h(t)$ of
\eqref{eq:disBoptcond_nom}-\eqref{eq:disBoptcond_ic}, the equation
\eqref{eq:disBoptcond_adj} is equivalent to the discrete adjoint scheme
\eqref{eq:aBDF} of the nominal BDF scheme.
\end{lemma}

\begin{proof}
Analogously to the beginning of the proof of Lemma \ref{lem:equiBDF},
each integral in \eqref{eq:disBoptcond_adj} is given by
\begin{align}
\int_{I_n} \dot{\vec{w}}^h(t)-&\vec{f}_{\vec{y}}(t,\vec{y}^h(t))\vec{w}^h(t) \,
\td \vec{\Lambda}^h(t) \nonumber\\
&= \vec{\lambda}_{n+1}^\T \left\lbrace \sum _{i=0}^{k_n}
\al{i}{n}\vec{w}_{n+1-i}-h_n
\vec{f}_{\vec{y}}(t_{n+1},\vec{y}_{n+1})\vec{w}_{n+1}\right\rbrace . \nonumber
\end{align}
Thus, equation \eqref{eq:disBoptcond_adj} can be formulated equivalently in
matrix form with $\vec{w}^\T := \left[\begin{array}{cccc} \vec{w}_1^\T &
\vec{w}_2^\T & \cdots & \vec{w}_N^\T \end{array}\right] \in \left(
\real^d\right)^N$
\begin{align} \label{eq:dvarf_adj}
&\left[\begin{array}{cccc} \vec{0} & \cdots & \vec{0} & J^\prime(\vec{y}_N)
\end{array}\right] \vec{w}-(\al{1}{0} \vec{\lambda}_1-\vec{l})^\T \vec{w}_0
\nonumber\\
&-\vec{\lambda}^\T \left[\mat{A} \otimes \mat{I}-
\left( \begin{array}{ccc} h_0 \vec{f}_{\vec{y}}(t_1,\vec{y}_1) & & \mat{0}\\
 & \ddots & \\ \mat{0} & &h_{N-1} \vec{f}_{\vec{y}}(t_N,\vec{y}_N)
\end{array} \right) \right]  \vec{w} =0, \; \forall \vec{w}_0,\vec{w}
\end{align}
which is linear in both the variations $\vec{w}_0,\vec{w}$ and the unknown
$\vec{\lambda}$. The equivalent time-stepping scheme goes backwards in time
starting with $J^\prime(\vec{y}_N) - \al{0}{N-1} \vec{\lambda}_N^\T +
h_{N-1}\vec{\lambda}_N^\T \vec{f}_{\vec{y}}(t_N,\vec{y}_N) =0$. Thus,
\eqref{eq:dvarf_adj} is equivalent
to \eqref{eq:aBDF} which finishes the proof.
\end{proof}

The necessary conditions for the well-posedness of
\eqref{eq:disBoptcond_nom}-\eqref{eq:disBoptcond_ic} are stated in numerous
textbooks on BDF methods, for example in \cite[Ch.4\S3]{Shampine1994}. With the
Lipschitz constant $L$ of $\vec{f}(t,\vec{y})$, the sequence of
stepsizes and orders has to satisfy $\big\vert h_n / \al{0}{n}\; L \big\vert <1$
in order to provide a unique solution $\vec{y}^h(t)$ of
\eqref{eq:disBoptcond_nom}-\eqref{eq:disBoptcond_ic}. The solution depends
continuously on the input data due to the stability of the integration scheme.
Since $\vec{f}_{\vec{y}}(t,\vec{y})$ is bounded by $L$ for all $(t,\vec{y})$
and $h_n$, $k_n$ satisfy $\big\vert h_n / \al{0}{n}\; L \big\vert <1$, the
matrix in \eqref{eq:dvarf_adj} is non-singular and thus
\eqref{eq:disBoptcond_adj} possesses a unique weak adjoint solution
$\vec{\Lambda}^h(t)$. The solution depends continuously on the input data
$J^\prime(\vec{y}_N)$ since the stability of the nominal integration scheme is
carried over to the discrete adjoint scheme \cite{Sandu2008}.
The well-posedness of \eqref{eq:disBoptcond_adj} can also be established using
the derivation of the equivalent scheme \eqref{eq:aBDF} by automatic
differentiation of \eqref{eq:BDF}, cf. Section \ref{sec:BDF_aBDF}.

\section{Convergence analysis of classical adjoints and weak adjoints}
\label{sec:conv}

In this section, we focus on the asymptotic behavior of the solutions of the
discrete adjoint scheme \eqref{eq:aBDF}. Therefore, we consider a nominal BDF
method of constant order $k$ with constant stepsizes $h$ using a self-starting
procedure for $\vec{y}_1,\dots,\vec{y}_m$ with $m\geq k-1$ fixed. We will call
this a \textit{non-adaptive BDF method}.
As seen in Section \ref{sec:aBDF}, the main part of the discrete adjoint
scheme, i.e. equation \eqref{eq:aBDF_main} with $n=N-k,\dots,m$, is a consistent
method of order $k$ for a variant of the adjoint equation \eqref{eq:aIVP}.
However, the adjoint initialization and termination steps do not give consistent
approximations. Nevertheless, we will prove that the
approximations in the main part converge linearly to the exact classical
solution $\vec{\lambda}(t)$ of \eqref{eq:aIVP} around the exact nominal solution
$\vec{y}(t)$. Using this result, we then show the strong
convergence of the finite element approximation $\vec{\Lambda}^h(t)$ towards the
solution $\vec{\Lambda}(t)$ of \eqref{eq:Boptcond_adj}, i.e. to the weak
solution of \eqref{eq:aIVP}, in the total variation norm of $\NBV[\ts,\tf]^d$.

\subsection{Convergence of the discrete adjoints to the classical adjoint}
\label{sec:Hconv}

The discrete adjoint scheme \eqref{eq:aBDF} of a non-adaptive BDF scheme reads
\bse\label{eq:nadp_aBDF}
\begin{align}
\alpha_0 \vec{\lambda}_N - J^\prime(\vec{y}_N)^\T &= h
\vec{f}^\T_{\vec{y}}(t_N,\vec{y}_N)\vec{\lambda}_N \label{eq:nadp_aBDF_ini}\\
\sum_{i=0}^{N-1-n} \alpha_i \vec{\lambda}_{n+1+i}
&=  h \vec{f}^\T_{\vec{y}}(t_{n+1},\vec{y}_{n+1})\vec{\lambda}_{n+1},\;
n=N-2,\dots, N-k  \label{eq:nadp_aBDF_start}\\
\sum_{i=0}^k \alpha_i \vec{\lambda}_{n+1+i}
&=  h \vec{f}^\T_{\vec{y}}(t_{n+1},\vec{y}_{n+1})\vec{\lambda}_{n+1},\;
n=N-k-1,\dots, m  \label{eq:nadp_aBDF_main}\\
\sum_{i=0}^k \al{i}{n+i} \vec{\lambda}_{n+1+i}
&=  h \vec{f}^\T_{\vec{y}}(t_{n+1},\vec{y}_{n+1})\vec{\lambda}_{n+1},\;
n=m-1,\dots, 0\label{eq:nadp_aBDF_term}
\end{align}
\ese
where \eqref{eq:nadp_aBDF_term} accounts for the nominal starting procedure.
To investigate the scheme \eqref{eq:nadp_aBDF} purely as an integration method
for the adjoint differential equation \eqref{eq:aIVP}, we consider a
continuously differentiable approximation $\tilde{\vec{y}}(t)$ satisfying
$\tilde{\vec{y}}(t_n)=\vec{y}_n$ for $n=0,\dots,N$, for example a quadratic
spline function interpolating $\left\lbrace \vec{y}_n\right\rbrace_{n=0}^N$ and
$\left\lbrace \vec{f}(t_n,\vec{y}_n)\right\rbrace_{n=0}^N$. With the adjoint
differential equation around $\tilde{\vec{y}}(t)$
\begin{align}\label{eq:app_aIVP}
\dot{\tilde{\vec{\lambda}}}(t)=-\vec{f}^\T_{\vec{y}}\left(t,\tilde{\vec{y}}
(t)\right) \tilde{\vec{\lambda}}(t),\;
\tilde{\vec{\lambda}}(\tf)=J^\prime\left(\tilde{\vec{y}}(\tf)\right)^\T
\end{align}
the main steps \eqref{eq:nadp_aBDF_main} can be seen as a BDF method of order
$k$ applied to \eqref{eq:app_aIVP}. The adjoint initialization steps 
\eqref{eq:nadp_aBDF_ini}-\eqref{eq:nadp_aBDF_start} can be interpreted as a
starting procedure for \eqref{eq:nadp_aBDF_main} giving inconsistent start
values $\vec{\lambda}_N,\dots,\vec{\lambda}_{N-k+1}$.

In the following, we study the asymptotic behavior for decreasing
$h\rightarrow 0$ and a fixed time point $t^*$ which belongs to
refining grids, i.e. for every stepsize $h$ there exists an $n=n(h)$ such that
$t^*=t_n$. The interval $[t_{m+1},t_{N-k}]$ of the main part of
\eqref{eq:nadp_aBDF} increases and approaches $(\ts,\tf)$ for $h\rightarrow
0$. By $\norm{\cdot}$ we denote any vector norm in $\real^d$.

\begin{lemma} \label{lem:Appconv}
Let $\vec{f}_{\vec{y}}(t,\tilde{\vec{y}}(t))$ be continuously differentiable
in $t\in[\ts,\tf]$ and $\tilde{\vec{y}}(t_n)=\vec{y}_n$ for $n=0,\dots,N$ where
$\{\vec{y}_n\}_{n=0}^N$ is computed by the non-adaptive BDF method of order $k$
with constant stepsize $h$. Let $\tilde{\vec{\lambda}}(t)$ be the exact solution
of
the adjoint differential equation \eqref{eq:app_aIVP} and let
$\{\vec{\lambda}_n\}_{n=1}^N$ be computed by \eqref{eq:nadp_aBDF}. Then, for a
fixed timepoint $t_n=t\in(\ts,\tf)$ there exists $H>0$ such that
\begin{align}
 \norm{\vec{\lambda}_n-\tilde{\vec{\lambda}}(t_n)}= \mO(h) \nonumber
\end{align}
as the grid is refined with $H>h\rightarrow 0$.
\end{lemma}

\begin{proof}
To ease the notion, we consider a scalar initial value problem, i.e. $d=1$.
Nevertheless, the proof is also valid for systems of initial value problems.
Furthermore, we define some abbreviations $B(t):=f^\T_y(t,\tilde{y}(t))$ and
$\eta:=J^\prime(\tilde{y}(\tf))^\T$. Thus, the 
starting procedure \eqref{eq:nadp_aBDF_ini}-\eqref{eq:nadp_aBDF_start}
can be written equivalently using $\vec{\lambda}^\T:=\left[
\begin{array}{lll}\lambda_N & \cdots &\lambda_{N-k+1}
\end{array}\right]$ and the $k\times 1$ unit vector $\vec{e}_1$
\begin{align}
\left[ \tilde{\mat{A}}-h\mat{B}(t_N, h)\right]\vec{\lambda} = \vec{e}_1
\eta \nonumber
\end{align}
where $\tilde{\mat{A}} = \bar{\mat{I}} \left[\mat{A}_{N-k+1:N,N-k+1:N}\right]
^\T \bar{\mat{I}}$ for the reverse identity matrix $\bar{\mat{I}}$ and
the matrix $\mat{A}$ from page \pageref{p:matA}, and
\begin{align}
\mat{B}(t_N,h):= \left(\begin{array}{cccc}
B(t_N) & &0\\
 & \ddots &  \\
0 && B(t_N-(k-1)h)
\end{array} \right)=
B(t_N) \mat{I} 
+ \mO(h)
\left(\begin{array}{cccc} 0& 0& \cdots &0 \\
0&1&&0\\
 \vdots&&\ddots &  \\
0 &0 && 1
\end{array} \right)\nonumber
\end{align}
using the Taylor series expansion of the entries $B(t_N-ih)$ around $t_N$. The
matrix $\tilde{\mat{A}}$ is nonsingular since $\alpha_0\neq 0$. Furthermore, for
$h$ small enough to satisfy $\nrm{h \tilde{\mat{A}}^{-1} \mat{B}(t_N,h)}{}<1$ we
can use the Neumann series to express the inverse of
$\mat{I}-h\tilde{\mat{A}}^{-1} \mat{B}(t_N,h)$, see for example \cite[Sec.
II.1]{Werner2000}, which yields
\begin{align} 
\vec{\lambda} &= \left[\tilde{\mat{A}}\left(\mat{I}
-h \tilde{\mat{A}}^{-1}\mat{B}(t_N,h) \right) \right]^{-1} \vec{e}_1\eta 
=\left\lbrace \sum_{j=0}^\infty \left( h
\tilde{\mat{A}}^{-1}\mat{B}(t_N,h)\right)^j\right\rbrace 
\tilde{\mat{A}}^{-1} \vec{e}_1\eta \nonumber\\
&=\left\lbrace \mat{i}+ h
\tilde{\mat{A}}^{-1}\mat{B}(t_N,h) + \mO(h^2)\right\rbrace 
\tilde{\mat{A}}^{-1} \vec{e}_1 \eta \nonumber\\
&= \tilde{\mat{A}}^{-1}\vec{e}_1 \eta
+ h \tilde{\mat{A}}^{-1}B(t_N)\tilde{\mat{A}}^{-1}\vec{e}_1 \eta
+\mO(h^2).\label{eq:startla}
\end{align}
We want to apply Theorem 4.3 of \cite{Henrici1963} to the linear
differential equation \eqref{eq:app_aIVP}. Note that the starting procedure
satisfies the assumptions of the theorem due to \eqref{eq:startla}. As BDF
methods are
strongly stable, the only essential root of the characteristic polynomial
$\rho(z)=\sum_{i=0}^k \alpha_i z^{k-i}$ is the principal root $z_1=1$. Thus, 
Theorem 4.3 of \cite{Henrici1963} gives for certain constants $K_1$ and $K_2$
\begin{align}
\lambda_n-\tilde{\lambda}(t_n) = \exp{\left( \int_{\tf}^{t_n} -B(\tau)\td
\tau\right) } \delta_1 + \theta \left(K_1+\frac{K_2}{t_n-h-\tf}\right) h
\nonumber
\end{align}
where $\abs{\theta}<1$ in the scalar case ($\nrm{\vec{\theta}}{}<1$ for
$d>1$). The quantity $\delta_1$ is
\begin{align}
\delta_1&:= \frac{1}{\rho^\prime(1)}\sum_{i=0}^{k-1} \gamma_i
(\lambda_{N-i}-\eta), \text{ where }\sum_{i=0}^{k-1} \gamma_i
z^i:=\frac{\rho(z)}{z-1}\nonumber
\end{align}
and the coefficients $\gamma_i$ sum up to 1, i.e. $\sum_{i=0}^{k-1}\gamma_i=1$.
The latter fact together with equation \eqref{eq:startla} gives for
$\vec{\gamma}^\T:=\left[ \begin{array}{lll}\gamma_0 & \cdots &
\gamma_{k-1}\end{array}\right]$
\begin{align}
\delta_1= \vec{\gamma}^\T \vec{\lambda} - \eta 
&= \vec{\gamma}^\T \left[\tilde{\mat{A}}^{-1}\vec{e}_1 \eta 
+ h \tilde{\mat{A}}^{-1}B(t_N)\tilde{\mat{A}}^{-1}\vec{e}_1 \eta
+\mO(h^2) \right] -\eta\nonumber \\
&= \left[ \vec{\gamma}^\T \tilde{\mat{A}}^{-1}\vec{e}_1 - 1 \right] \eta
+ h \vec{\gamma}^\T
\tilde{\mat{A}}^{-1}B(t_N)\tilde{\mat{A}}^{-1}\vec{e}_1 \eta +\mO(h^2).
\nonumber
\end{align}
The coefficient $ \vec{\gamma}^\T \tilde{\mat{A}}^{-1}\vec{e}_1 - 1$ of the
first addend vanishes which can be verified easily for all BDF methods up to
order 6. Thus, we obtain
\begin{align}
\lambda_n-\tilde{\lambda}(t_n) 
= h \exp{\left( \int^{\tf}_{t_n} B(\tau)\td \tau\right) } 
\vec{\gamma}^\T \tilde{\mat{A}}^{-1}B(t_N)\tilde{\mat{A}}^{-1}\vec{e}_1 \eta&
\nonumber\\
+ h\, \theta \left(K_1+\frac{K_2}{t_n-h-\tf}\right) &+\mO(h^2)
\nonumber
\end{align}
where both coefficients are bounded which proves the assertion. 
\end{proof}

The main result of this subsection is the following.

\begin{theorem}\label{theo:Hconv}
Let $\vec{f}(t,\vec{y})$ be continuously differentiable with respect
to $(t,\vec{y})$. Let $\vec{\lambda}(t)$ be the exact solution of the adjoint
differential equation \eqref{eq:aIVP} and let $\{\vec{\lambda}_n\}_{n=1}^N$ be
computed by \eqref{eq:nadp_aBDF}. Then, for a fixed timepoint
$t_n=t\in(\ts,\tf)$ there exists $H>0$ such that
\begin{align}
 \norm{ \vec{\lambda}_n-\vec{\lambda}(t_n)} = \mO(h) 
\end{align}
as the grid is refined with $H>h\rightarrow 0$.
\end{theorem}

\begin{proof}
Let the continuously differentiable spline $\tilde{\vec{y}}(t)$ be composed
of quadratic polynomials on $I_n$ such that $\tilde{\vec{y}}(t_n)=\vec{y}_n$,
$\tilde{\vec{y}}(t_{n+1})=\vec{y}_{n+1}$ and
$\dot{\tilde{\vec{y}}}(t_{n+1})=\vec{f}(t_{n+1},\vec{y}_{n+1})$ for
$n=0,\dots,N-1$. Furthermore, we define the interpolation operator $\op{\mI}$
that maps a continuously differentiable function $\vec{g}(t)$ to a continuously
differentiable spline $\op{\mI}\vec{g}(t)$ that is composed of quadratic
polynomials on $I_n$ with $\op{\mI}\vec{g}(t_n)=\vec{g}(t_n)$,
$\op{\mI}\vec{g}(t_{n+1})=\vec{g}(t_{n+1})$ and
$\dot{\op{\mI}\vec{g}}(t_{n+1})=\dot{\vec{g}}(t_{n+1})$ for
$n=0,\dots,N-1$. Then, the difference of $\tilde{\vec{y}}(t)$ and
$\op{\mI}\vec{y}(t)$ in $C^0$-norm is
\begin{align}
\nrm{\tilde{\vec{y}}(t)-\op{\mI}\vec{y}(t)}{C^0[\ts,\tf]^d}=\mO(h)\nonumber
\end{align}
using Taylor expansions and the convergence of the nominal BDF method. Due to
the assumption on $\vec{f}(t,\vec{y})$, the exact nominal solution $\vec{y}(t)$
of \eqref{eq:IVP} is twice continuously differentiable such that
\begin{align}
\nrm{\vec{y}(t)-\op{\mI}\vec{y}(t)}{C^0[\ts,\tf]^d}=\mO(h^2)\nonumber
\end{align}
due to the approximation property of quadratic splines. Thus, it is
\begin{align} \label{eq:spline_esol}
\nrm{\tilde{\vec{y}}(t)-\vec{y}(t)}{C^0[\ts,\tf]^d} \leq
\nrm{\tilde{\vec{y}}(t)-\op{\mI}\vec{y}(t)}{C^0} +
\nrm{\op{\mI}\vec{y}(t)-\vec{y}(t)}{C^0} = \mO(h).
\end{align}
Since both adjoint differential equations \eqref{eq:aIVP} and
\eqref{eq:app_aIVP} are linear, their solutions $\vec{\lambda}(t)$ and
$\tilde{\vec{\lambda}}(t)$ can be given explicitly. Substracting the exact
adjoint solutions and using \eqref{eq:spline_esol} yields in the $C^0$-norm
\begin{align}\label{eq:laTilde_la}
\nrm{\tilde{\vec{\lambda}}(t)-\vec{\lambda}(t)}{C^0[\ts,\tf]^d} =
\mO(h)
\end{align}
which implies directly the pointwise convergence for every $t\in[\ts,\tf]$.
Thus, together with Lemma \ref{lem:Appconv} we obtain
\begin{align}
\norm{\vec{\lambda}_n-\vec{\lambda}(t_n)}\leq
\norm{\vec{\lambda}_n-\tilde{\vec{\lambda}}(t_n)} +
\norm{\tilde{\vec{\lambda}}(t_n)-\vec{\lambda}(t_n)} = \mO(h)\nonumber
\end{align}
for $t_n\in (\ts,\tf)$.
\end{proof}

\begin{remark}
If $\vec{f}(t,\vec{y})$ is $k$-times continuously differentiable in
$(t,\vec{y})$, the start errors of the nominal BDF method of order $k$ are small
enough (i.e. the convergence of order $k$ is guaranteed), and the spline is of
corresponding order, then \eqref{eq:laTilde_la} holds with order $k$ in $h$.
\end{remark}

The discrete adjoints resulting from the adjoint initialization and termination
steps differ from the exact adjoints in a constant way. For $n=N,\dots,N-k+1$
the difference is bounded by a positive constant $c_n$ times the state
$\vec{\lambda}(\tf)=J^\prime(\vec{y}(\tf)^\T $, i.e.
\begin{align}
\norm{\vec{\lambda}_n-\vec{\lambda}(t_n)} \leq c_n
\norm{J^\prime(\vec{y}(\tf))}+\mO(h).\nonumber
\end{align}
This can be shown using \eqref{eq:laTilde_la}, the Taylor expansion of
$\tilde{\vec{\lambda}}(t_n)$ around $\tf$ and the Neumann series of the inverse
of  $\al{0}{n}\mat{I}-h\vec{f}_{\vec{y}}(t_{n+1},\vec{y}_{n+1})$.
For the discrete adjoints from the adjoint termination steps
\eqref{eq:nadp_aBDF_term}, one also needs Lemma \ref{lem:Appconv} and obtains a
multiple of $\vec{\lambda}(\ts)$.

Without modifications of the adjoint initialization steps
\eqref{eq:nadp_aBDF_ini}-\eqref{eq:nadp_aBDF_start}, the discrete adjoints on
the main part converge linearly to the exact adjoint solution
$\vec{\lambda}(t)$ of \eqref{eq:aIVP}. Nevertheless, we still have to consider
the oscillations of the discrete adjoints at the interval ends of $[\ts,\tf]$
which are due to the
inconsistency of the adjoint initialization and termination steps. We will
do this in the next section.

\subsection{Convergence of the finite element approximation to the weak adjoint}
\label{sec:Bconv}

We will prove the convergence of the finite element approximation of the weak
adjoint to the exact weak adjoint of \eqref{eq:aIVP} given by
\eqref{eq:Boptcond_adj} with respect to the total variation norm of
$\NBV[\ts,\tf]^d$ (i.e. strong convergence).

\begin{theorem}\label{theo:conv}
The finite element approximation $\vec{\Lambda}^h(t)=\sum_{n=1}^N
h_{n-1}\vec{\lambda}_n H_n(t)$ given by the discrete adjoint scheme
\eqref{eq:nadp_aBDF} of a non-adaptive BDF method of constant
order $k$ with constant stepsize $h$ converges to the exact weak adjoint
solution
$\vec{\Lambda}(t)=\int_\ts^t \vec{\lambda}(\tau)\td \tau$
where $\vec{\lambda}(\tau)$ solves \eqref{eq:aIVP}. The convergence is with
respect to the total variation norm of $\NBV[\ts,\tf]^d$.
\end{theorem}

\begin{proof} 
Let $h:=\frac{\tf-\ts}{N}$ be the stepsize of the equidistant grid. Thus, the
nodes are $t_n=\ts+nh$ for $n=0,\dots, N$. We use the norms mentioned in Section
\ref{sec:gCons} and consider firstly the $i$th component, $1\leq i \leq d$. To
ease the notion, we set $\Lambda:=\vec{\Lambda}_i$,
$\Lambda^h:=\vec{\Lambda}_i^h$, $g:=\vec{g}_i$ such that the dual norm reads
\begin{align}
\nrm{\Lambda-\Lambda^h}{\NBV[\ts,\tf]}=
\sup_{\nrm{g}{C^0[\ts,\tf]}=1} \abs{\int_{\ts}^{\tf}
g(t)\td \left(\Lambda-\Lambda^h \right)(t) }.\nonumber
\end{align}
As $\Lambda$ is given by $\Lambda(t)=\int_\ts^t \lambda(\tau)\td
\tau$ and $\Lambda^h$ is a jump function it holds \cite[Sec.36
Example 3]{Kolmogorov1970} %
\begin{align}
\int_{\ts}^{\tf}
g(t)\td \left(\Lambda-\Lambda^h \right)(t) =
\int_\ts^\tf \lambda(t)g(t) \td t -
\sum_{n=1}^N h \lambda_n g(t_n).\nonumber
\end{align}
Approximating the integral by the composite trapezoidal rule for
equidistant grids yields
\begin{align}
h\left\lbrace \frac{1}{2} \lambda(t_0) g(t_0)+
\sum_{n=1}^{N-1}\lambda(t_n)g(t_n)+\frac{1}{2} \lambda(t_N)g(t_N) \right\rbrace 
+ \mO(h^2)- \sum_{n=1}^N h\lambda_n g(t_n)\nonumber\\
= h \left\lbrace \frac{1}{2} \lambda(t_0)g(t_0) +
\sum_{n=1}^{N}\left[ \lambda(t_n)-\lambda_n\right] g(t_n)
- \frac{1}{2} \lambda(t_N) g(t_N) \right\rbrace + \mO(h^2).
\nonumber
\end{align}
We obtain a bound for the $\NBV[\ts,\tf]^d$-dual norm of
$\vec{\Lambda}-\vec{\Lambda}^h$ by taking the absolute value, using the triangle
inequality and the fact that
$\nrm{g}{C^0[\ts,\tf]}=1$, i.e.
\begin{align}
\nrm{\Lambda-\Lambda^h}{\NBV[\ts,\tf]}\leq h \left\lbrace \abs{\lambda(t_0)}+
\sum_{n=1}^{N} \abs{\lambda(t_n)-\lambda_n }
+ \abs{\lambda(t_N)} \right\rbrace + \mO(h^2). \nonumber
\end{align}
With Theorem \ref{theo:Hconv} the sum over the main part becomes
\begin{align}
\sum_{n=m+1}^{N-k} \abs{\lambda(t_n)-\lambda_n} =
\sum_{n=m+1}^{N-k} \mO(h)=\mO(1) \nonumber
\end{align}
such that the norm is bounded by
\begin{align}
&\nrm{\Lambda-\Lambda^h}{\NBV[\ts,\tf]} \nonumber\\
&\leq 
h \left\lbrace \abs{\lambda(t_0)} + \sum_{n=1}^{m}
\abs{\lambda(t_n)-\lambda_n} + \mO(1)+\sum_{n=1}^{k-1}
\abs{\lambda(t_{N-n})-\lambda_{N-n}}
+ \abs{\lambda(t_N)} \right\rbrace + \mO(h^2). \nonumber
\end{align}
Since the magnitude of all remaining addends is bounded according to the end of
Section \ref{sec:Hconv} and their number is independent of the step number $N$,
it is $\nrm{\Lambda-\Lambda^h}{\NBV[\ts,\tf]} = \mO(h)$.
As this holds for all $i=1,\dots,d$ and the dual norm coincides with the total
variation norm (cf. Section \ref{sec:gCons}), the assertion is shown.
\end{proof}

By small modifications in the proof of Theorem \ref{theo:conv}, the assertion
can be widened to variable stepsizes in the starting procedure.

The uniform convergence in the total variation norm of $\NBV[\ts,\tf]^d$ implies
the pointwise convergence on the entire time interval which can be shown by
utilizing the particular partition $\{\ts,\theta,\tf\}$ for an arbitrary time
point $\theta \in [\ts,\tf]$. Thus, Theorem \ref{theo:conv} implies the
pointwise convergence of $\vec{\Lambda}^h(t)$ to $\vec{\Lambda}(t)$ on the
entire time interval at least with the same convergence rate.

\section{Numerical results}
\label{sec:numres}

We illustrate the theoretical results with the help of a nonlinear test case
with analytic nominal and adjoint solutions. The Catenary
\cite[p.15]{Hairer1993} is given by a second-order ODE
\begin{align}
\ddot{y}(t)= p\sqrt{1+\dot{y}(t)^2},\quad p>0.\nonumber
\end{align}
We reformulate the initial value problem as system of first-order equations 
\begin{align}
\dot{y}_1(t) &= y_2(t)\nonumber\\
\dot{y}_2(t) &= p \sqrt{1+y_2(t)^2}\nonumber
\end{align}
and solve it on the interval $[0,2]$ for $p=3$ and $\vec{y}(0)=[1/3 \cosh(-3)
\enskip \sinh( -3 )]^\T$. As criterion of interest we choose
$J(\vec{y}(2))=y_1(2)$. The analytic nominal solution is
\begin{align}
\vec{y}(t) = 
\left( \begin{array}{c} B + \frac{1}{p}\cosh(pt+A)\\
\sinh(pt+A) \end{array} \right) \nonumber
\end{align}
and the analytic weak adjoint solution in the space $\NBV[\ts,\tf]^2$ is
\begin{align}\label{eq:anaASol}
\vec{\Lambda}(t)= 
\left( \begin{array}{c} t \\ 
-\frac{1}{p^2} \ln ( \cosh( pt+A ) ) + \frac{2}{p^2}\sinh( p\tf+A )
\arctan\left( \e^{ pt+A } \right) 
\end{array} \right)
\end{align}
where $A$ and $B$ are determined implicitly by the initial values.

\subsection{Non-adaptive BDF method}

We consider a non-adaptive BDF method of constant order $2$ on an equidistant
grid with stepsize $h$. The self-starting procedure consists of two first-order
BDF steps with stepsize $h/2$. The simulations are performed in Matlab.

The lower row of Figure \ref{fig:BDF2_discrAdj} compares the discrete adjoints
for two different stepsizes $h=2^{-4}$ and $h=2^{-6}$ to the analytic solution
of the adjoint differential equation. The oscillations of the discrete
adjoints at the interval ends are due to the inconsistency of the adjoint
initialization and termination steps of the discrete
adjoint scheme with the adjoint differential equation (cf. Section
\ref{sec:aBDF}). Nevertheless, the discrete adjoints converge on the
open interval $(0,2)$ towards the analytic adjoint solution as proven by
Theorem \ref{theo:Hconv}. In the upper row of Figure
\ref{fig:BDF2_discrAdj} the finite element approximation $\vec{\Lambda}^h(t)$ is
compared to the weak adjoint $\vec{\Lambda}(t)$ given by
\eqref{eq:anaASol}. It converges on the whole time interval as shown by Theorem
\ref{theo:conv}.

\begin{figure}
\centering
\subfigure{\includegraphics[width = 0.47\textwidth]{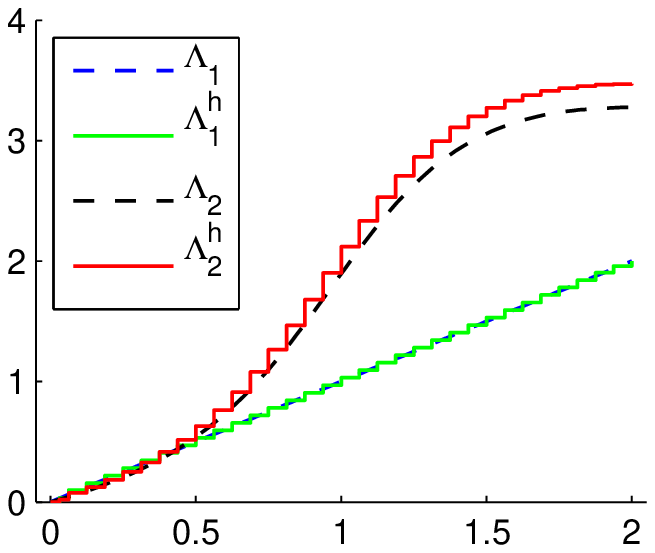}}
\quad
\subfigure{\includegraphics[width = 0.47\textwidth]{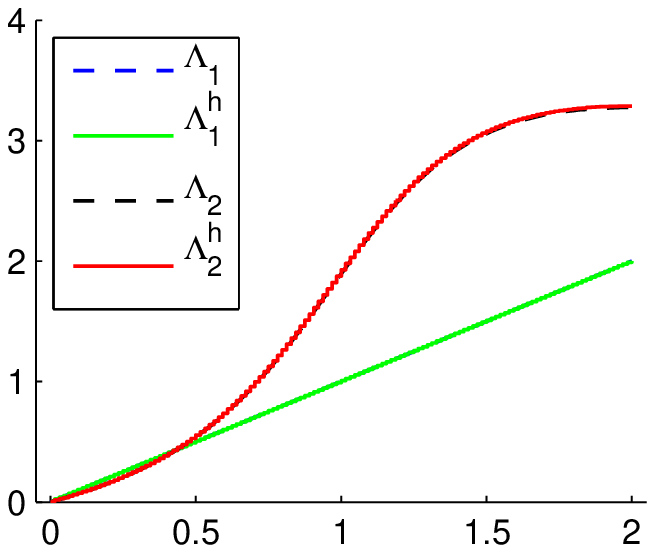}}\\
\addtocounter{subfigure}{-2}
\subfigure[$h= 2^{-4}$]{\includegraphics[width =
0.47\textwidth]{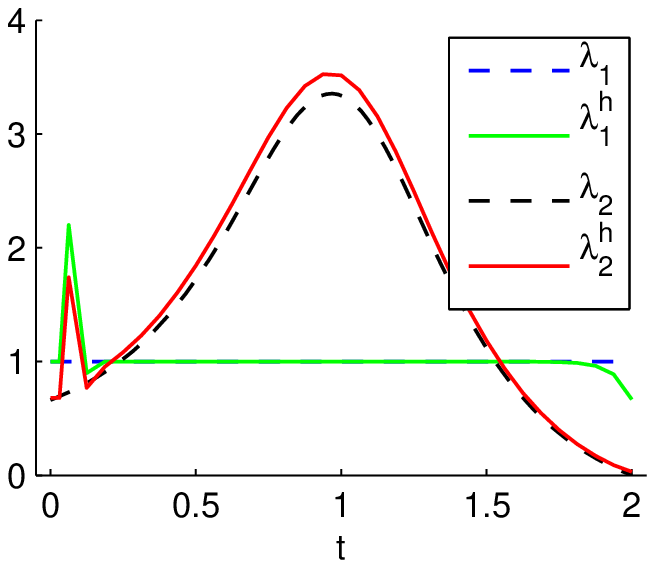}}
\quad
\subfigure[$h= 2^{-6}$]{\includegraphics[width =
0.47\textwidth]{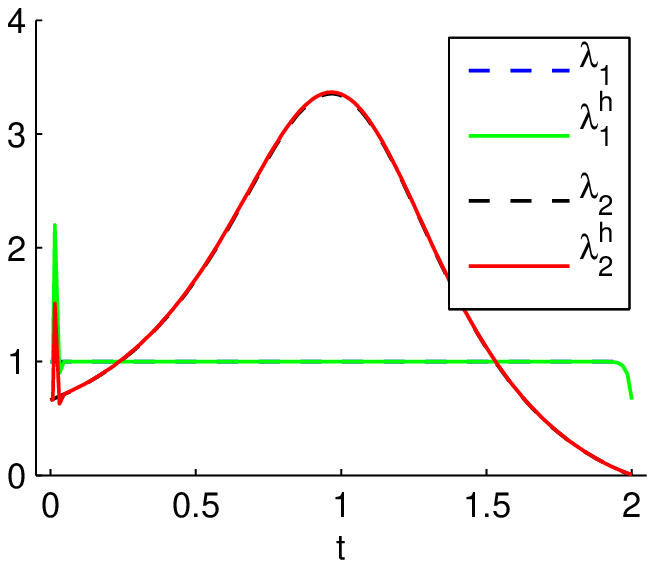}}
\caption{Results of the non-adaptive BDF method for two different
stepsizes. Comparison of the finite element approximation of the weak adjoint
and the analytic weak adjoint (top) as well as the discrete adjoints in
comparison to analytic Hilbert space adjoint (bottom) for different stepsizes.}
\label{fig:BDF2_discrAdj}
\end{figure}

Figure \ref{fig:conv} shows the Euclidean norm of the difference between the
analytic weak adjoint \eqref{eq:anaASol} and the finite element approximation,
i.e.
\begin{align}
 \text{Error} = \nrm{\vec{\Lambda}(t)-\vec{\Lambda}^h(t)}{2}, \nonumber
\end{align}
evaluated at the final time $t=\tf =2$ and at some interior time point $t=1.25$,
respectively, for shrinking stepsizes. The error evaluated at the final time
decreases at second order rate, a somewhat better behavior than predicted by
the convergence theory of Section \ref{sec:Bconv}. This might be due to the
second order convergence of the discrete adjoints at the initial time together
with a possible cancellation of discrepancies of the discrete adjoints at the
interval ends (depicted in the lower row of Figure \ref{fig:BDF2_discrAdj}).
Overall, this observation calls for a closer theoretical
investigation. The error at the interior time point $t=1.25$ shows the expected
linear convergence, cf. Theorem \ref{theo:conv} and the subsequent comment on
the pointwise convergence.

\begin{figure}
\centering
\includegraphics[width = 0.5\textwidth]{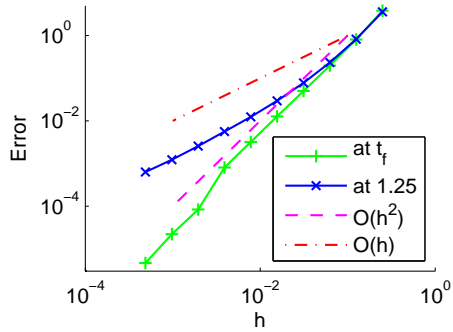}
\caption{Convergence of the finite element approximation of the weak adjoint to
the analytic weak adjoint. Error evaluated at the final time $\tf = 2$ and at
the interior time point $t=1.25$.}
\label{fig:conv}
\end{figure}

\subsection{Adaptive BDF method}

The software package \texttt{DAESOL-II} \cite{Albersmeyer2010d} provides an
efficient realization of a variable-order variable-stepsize BDF method based on
a sophisticated order and stepsize selection. Furthermore, it contains efficient
ways to compute the discrete adjoints
\cite{Albersmeyer2006,Albersmeyer2009,Albersmeyer2010d}.
We solved the Catenary for two different accuracies (relative tolerance
$10^{-4}$ and $10^{-9}$) to get a first asymptotic impression of the finite
element approximation of the adjoint in the case of fully adaptive BDF methods.
The results are depicted in Figure \ref{fig:adapt_discrAdj}.

In areas of constant BDF order (fourth row of Figure \ref{fig:adapt_discrAdj})
and constant stepsizes (third row), the discrete adjoints converge to the
analytic
adjoint solution (second row) as seen in the right column on the interval
$(1,1.7)$ approximately. On the other areas, i.e. where the order is varying and
stepsize is changing, the discrete adjoints are highly oscillating (second row).
Nevertheless, also in these cases, the finite element approximations
$\vec{\Lambda}^h(t)$ converge to the analytic weak adjoint solution
\eqref{eq:anaASol} on the entire time interval (first row of Figure
\ref{fig:adapt_discrAdj}).

\begin{figure}
\centering
\subfigure{\includegraphics[width =
0.47\textwidth]{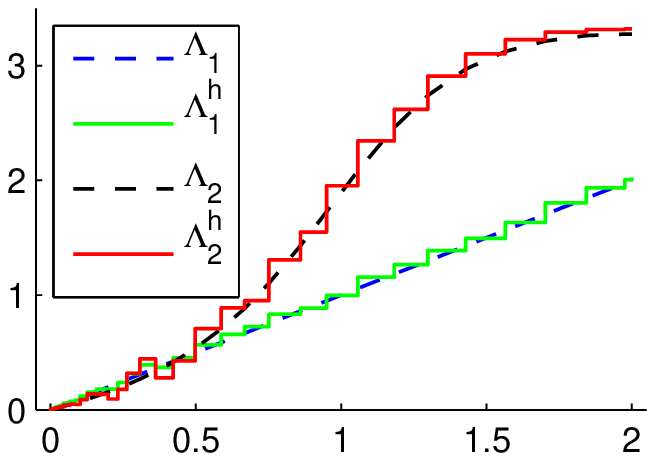}}
\quad
\subfigure{\includegraphics[width =
0.47\textwidth]{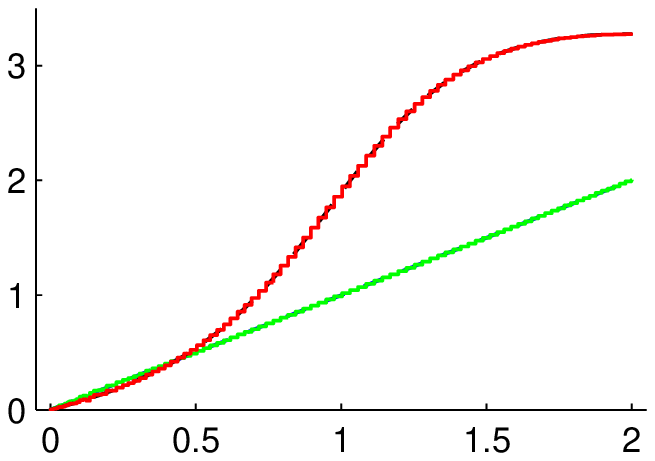}}\\
\subfigure{\includegraphics[width =
0.47\textwidth]{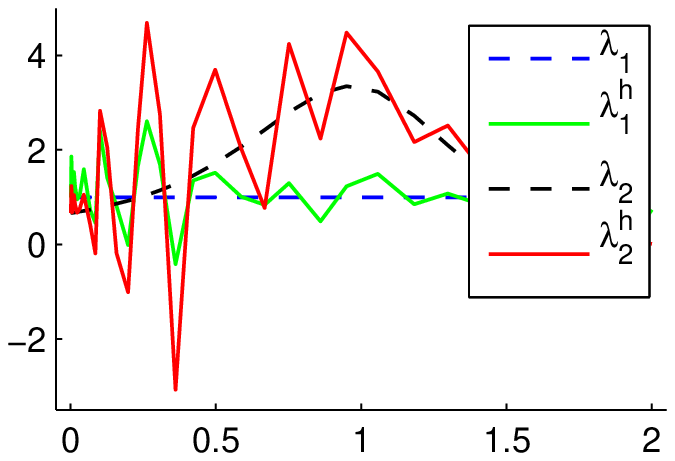}}
\quad
\subfigure{\includegraphics[width =
0.47\textwidth]{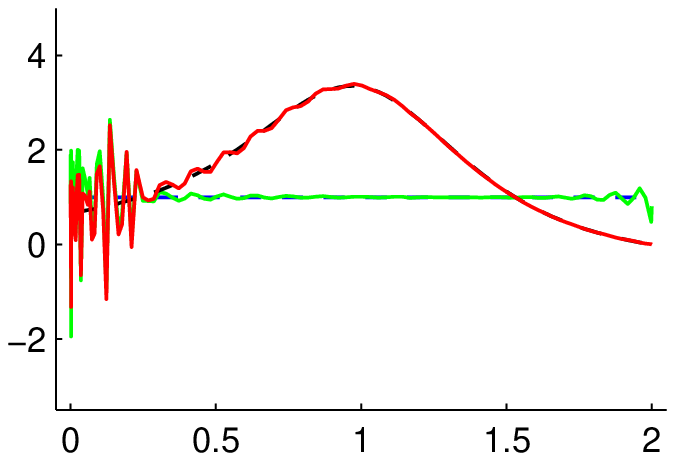}}\\
\subfigure{\includegraphics[width =
0.47\textwidth]{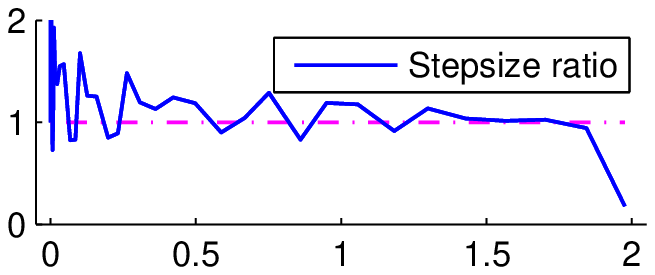}}
\quad
\subfigure{\includegraphics[width =
0.47\textwidth]{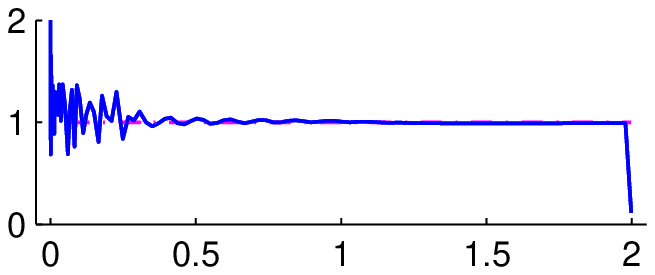}}\\
\addtocounter{subfigure}{-6}
\subfigure[Relative tolerance $10^{-4}$]{\includegraphics[width =
0.47\textwidth]{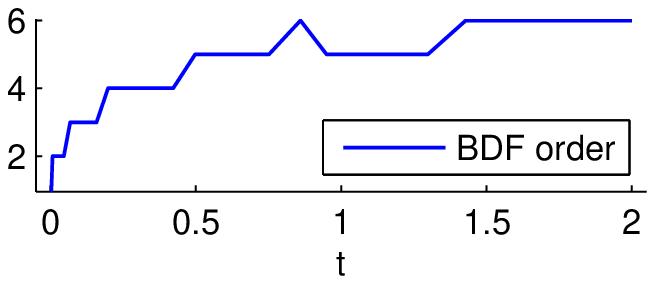}}
\quad
\subfigure[Relative tolerance $10^{-9}$]{\includegraphics[width =
0.47\textwidth]{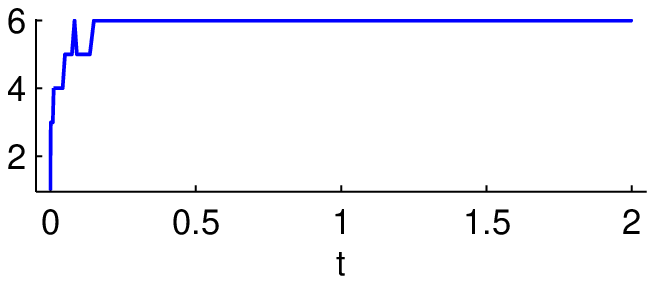}}
\caption{Results of the adaptive BDF method for different accuracies. Comparison
of the finite element approximation of the weak adjoint and the analytic weak
adjoint (top) as well as the discrete adjoints in comparison to analytic Hilbert
space adjoints (second row). Stepsize ratio (third row) and BDF order (bottom)
of the integration scheme.}
\label{fig:adapt_discrAdj}
\end{figure}

\section{Summary and outlook}

In this contribution, we have addressed the issue of relating the discrete
adjoints of variable-order variable-stepsize BDF methods to the solution of the
adjoint differential equation \eqref{eq:aIVP}. Since for multistep methods the
common Hilbert space setting is not appropriate to interpret the discrete
adjoints, we have developed a new Banach space approach. It is based on a
constrained variational problem in the space of all continuously differentiable
functions with Lagrange multiplier in the space of all normalized functions of
bounded variation. We have approximated the infinite-dimensional optimality
conditions by a Petrov-Galerkin discretization and have shown the equivalence of
the resulting equations to the BDF scheme and its discrete adjoint scheme
obtained by adjoint internal numerical differentiation. Thus, discretization and
optimization commute in the presented framework and the finite element
approximation of the weak adjoint is obtained by a simple post-processing of the
discrete
adjoints. Furthermore, we have demonstrated that the discrete adjoint scheme of
a non-adaptive BDF method produces discrete adjoints which converge linearly to
the solution of \eqref{eq:aIVP} on the inner time interval although the adjoint
initialization steps are inconsistent. We have used this
result to prove the linear convergence of the finite element approximation on
the entire time interval to the weak adjoint solution of \eqref{eq:aIVP} in the
space of normalized functions of bounded variation. 

The theoretical results have been observed numerically using a non-adaptive BDF
method to solve the Catenary. Additionally, we have given numerical evidence
that the
finite element approximation serves as proper quantity to approximate the weak
adjoint also in the case of fully adaptive BDF methods, i.e. also in areas of
variable order and variable stepsize. 

Thus, we now have a quantity at hand which can be used within global error
estimation techniques. The functional-analytic framework allows to
carry over estimation techniques from finite element methods to BDF methods.
Furthermore, the approximations to the weak adjoints can now be computed
efficiently and accurately by automatic differentiation of the efficient
variable-order variable-stepsize BDF method without the need of explicit
derivation of the adjoint equations.

\begin{acknowledgements}
The authors express their gratitude to Christian Kirches and \mbox{Andreas}
Potschka for valuable discussions on the subject. Scientific support of the
DFG-Graduate-School 220 ``Heidelberg Graduate School of Mathematical and
Computational Methods for the Sciences'' is gratefully acknowledged. Funding
graciously provided by the German Ministry of Education and Research (Grant ID:
03MS649A), and the Helmholtz association through the SBCancer programme.
The research leading to these results has received funding from the European
Union Seventh Framework Programme FP7/2007-2013 under grant agreement
n\textsuperscript{o}~FP7-ICT-2009-4 248940.
\end{acknowledgements}

\bibliographystyle{spmpsci}

\end{document}